\input ./style/arxiv-general.cfg
\documentclass[aop,MSNbibl,seceqn,dvips]{arximspdf}
\makeatletter
   \@ifpackageloaded{graphicx}{}{\usepackage{graphicx}}
\makeatother
\usepackage{mathrsfs}

\doi{10.1214/14-AOP925} %kopijuoti is PTS
\volume{43}
\issue{4}
\pubyear{2015}
\firstpage{1944}
\lastpage{1991}
\docsubty{FLA}

\makeatletter
\newcommand{\lleft}{\left}
\newcommand{\rrvert}{\vert}
\newcommand{\rright}{\right}
\newcommand{\rrVert}{\Vert}
\newcommand{\llvert}{\vert}
\newcommand{\llVert}{\Vert}
\newcommand{\nicefrac}[2]{{#1}/{#2}}
\newcommand{\R}{\mathbf{R}}
\newcommand{\lip}{\mathrm{L}_\sigma}
\renewcommand{\P}{\mathrm{P}}
\newcommand{\E}{\mathrm{E}}
\newcommand{\F}{\mathscr{F}}
\newcommand{\1}{\mathbf{1}}
\renewcommand{\d}{\mathrm{d}}
\newcommand{\e}{\mathrm{e}}
\renewcommand{\Re}{\operatorname{Re}}
\newcommand{\mod}{\operatorname{mod}}
\newtheorem{stat}{Statement}[section]
\newtheorem{proposition}[stat]{Proposition}
\newtheorem{corollary}[stat]{Corollary}
\newtheorem{theorem}[stat]{Theorem}
\newtheorem{lemma}[stat]{Lemma}
\newproclaim{definition}[stat]{Definition}
\newproclaim{remark}[stat]{Remark}
\newproclaim{example}[stat]{Example}
\makeatother

\begin{document}
\begin{frontmatter}

\title{Nonlinear noise excitation of intermittent stochastic PDEs and the topology of LCA groups\thanksref{T1}}
\runtitle{Nonlinear noise excitation}

\begin{aug}
\author[A]{\fnms{Davar}~\snm{Khoshnevisan}\ead[label=e1]{davar@math.utah.edu}}
\and
\author[A]{\fnms{Kunwoo}~\snm{Kim}\corref{}\ead[label=e2]{kkim@math.utah.edu}}
\runauthor{D. Khoshnevisan and K. Kim}
\affiliation{University of Utah}
%\dedicated{}
\address[A]{Department of Mathematics\\
University of Utah\\
Salt Lake City, Utah 84112-0090\\
USA\\
\printead{e1}\\
\phantom{E-mail: }\printead*{e2}} %adresu isvedimo komanda gale!
%\address{}
\end{aug}
\thankstext{T1}{Supported in part by the NSF Grants DMS-10-06903 and
DMS-13-07470.}

% HISTORY:
\received{\smonth{3} \syear{2013}}
\revised{\smonth{2} \syear{2014}}
%\accepted{\smonth{} \syear{}}

% ABSTRACT
%
\begin{abstract}
Consider the stochastic heat equation
$\partial_t u = \mathscr{L}u + \lambda\sigma(u)\xi$, where
$\mathscr{L}$ denotes the generator of a L\'evy process
on a locally compact Hausdorff Abelian group $G$, $\sigma\dvtx \mathbf{R}\to\mathbf{R}$
is Lipschitz continuous, $\lambda\gg1$ is a large
parameter, and $\xi$ denotes space--time white noise
on $\mathbf{R}_+\times G$.

The main result of this paper contains a near-dichotomy for the
(expected squared) energy
$\mathrm{E}(\| u_t\| _{L^2(G)}^2)$ of the solution. Roughly speaking,
that dichotomy says that, in all known cases where $u$ is
intermittent, the energy of the solution behaves generically as
$\exp\{\operatorname{const}\cdot\,\lambda^2\}$ when $G$ is discrete 
and~$\ge\exp\{\operatorname{const}\cdot\,\lambda^4\}$ when $G$ is connected.
\end{abstract}

% KEYWORDS
\end{frontmatter}

\setcounter{footnote}{1}
%s1 #&#
\section{An informal introduction}\label{sec1}

Consider a stochastic heat equation of the form
{\renewcommand{\theequation}{SHE}
%e1.1 #&#
\begin{equation}\label{SHE}
\frac{\partial}{\partial t} u = \mathscr{L} u + \lambda\sigma(u) \xi.
\end{equation}}%
Here, $\sigma\dvtx \R\to\R$ is a Lipschitz continuous function,
$t>0$ denotes the time variable,
$x\in G$ is the space variable, for a locally compact Hausdorff Abelian group
$G$---such as $\R$, $\mathbf{Z}^d$, or
$[0,1]$---and the initial value $u_0\dvtx G\to\R$ is nonrandom and well behaved.
The operator $\mathscr{L}$ acts on the variable $x$
only, and denotes the generator of a L\'evy process on $G$,
and $\xi$ denotes space--time white noise on $(0,\infty)\times G$
whose control
measure is the restriction of the Haar measure on $\R\times G$ to
$(0,\infty)\times G$.
The number $\lambda$ is a positive parameter; this is the so-called
\emph{level of the noise}.

\setcounter{equation}{0}

In this paper, we study the ``noisy case.'' That is when $\lambda$ is
a large quantity. The case that $\lambda$ is small is also interesting;
see, for example, the deep theory of Freidlin and Wentzel \cite{Freidlin}.

We will consider only examples of (\ref{SHE}) that are \emph{intermittent}.
Intuitively speaking, ``intermittency'' is the property that the
solution $u_t(x)$
develops extreme oscillations at some values of $x$, typically when $t$
is large. Intermittency
was announced first (1949) by Batchelor and Townsend in a WHO
conference in Vienna~\cite{BatchelorTownsend}, and slightly later by Emmons~\cite{Emmons}
in the context of
boundary-layer turbulence. Ever since that time, intermittency
has been observed in an enormous number of scientific disciplines. Shortly,
we will point to concrete instances in theoretical physics. In the mean time,
let us also mention that, in neuroscience, intermittency is observed as
``spikes''
in neural activity. (Tuckwell \cite{Tuckwell} contains
a gentle introduction to SPDEs in neuroscience.) And in finance, intermittency
is usually associated with financial ``shocks.''

The standard mathematical definition of intermittency
(see Molchanov \cite{Molch91} and Zeldovich et al. \cite{ZRS})
is that
%
%
%e1.1 #&#
\begin{equation}
\frac{\gamma(k)}{k} < \frac{\gamma(k')}{k'}\qquad\mbox{whenever }2\le
k<k'<\infty,
\end{equation}
where $\gamma$ denotes any reasonable choice of a so-called
Lyapunov exponent of the moments of the energy of the solution:
we may use either
\[
\gamma(k):= \limsup_{t\to\infty} t^{-1}\log\E \bigl(
\llVert u_t\rrVert_{L^2(G)}^k \bigr)\quad\mbox{or}
\quad\gamma(k):=\liminf_{t\to\infty} t^{-1}\log\E \bigl(
\llVert u_t\rrVert_{L^2(G)}^k \bigr).
\]
Other essentially-equivalent choices are also possible.
One can justify this definition either by making informal analogies
with finite-dimensional nonrandom dynamical systems \cite{Mandelbrot},
or by making a somewhat informal appeal to the
Borel--Cantelli lemma \cite{BC}.
Gibbon and Titi \cite{GibbonTiti} contains an exciting modern account
of mathematical intermittency and its role in our present-day understanding
of physical intermittency.

In the case that $G=\R$, $G=[0,1]$ or $G=\mathbf{Z}^d$,
there is a huge literature that is devoted to the intermittency
properties of
(\ref{SHE}) when $\sigma(x)=\operatorname{const}\cdot\, x$; this particular
model---the so-called
\emph{parabolic Anderson model}---is interesting in its own right, as
it is
connected deeply with a large number of diverse
questions in probability theory and mathematical physics.
See, for example, the ample bibliographies
of \cite
{BC,CM94,Corwin,CM,DoeringSvavov,denHollander,FoondunKh,GartnerKonig,Kardar,KPZ,Molch91,ZRS}.

The parabolic Anderson model arises in a surprisingly large number of
diverse scientific problems; see Carmona and Molchanov \cite{CM94},
Introduction.
We mention quickly a few such instances:
if $\sigma(0)=0$, $u_0(x)>0$ for all $x\in G$, and $G$ is either $\R$
or $[0,1]$ then
Mueller's comparison principle \cite{Mueller} shows that $u_t(x)> 0$
almost surely for all $t>0$ and $x\in G$; see also \cite{Minicourse}, page~130.
In that case, $h_t(x):=\log u_t(x)$ is well defined
and is the so-called Cole--Hopf solution to the KPZ equation
of statistical mechanics \cite{Kardar,KPZ}. The parabolic
Anderson model has many connections also with the stochastic Burger's equation
\cite{CM94} and Majda's model of shear-layer flow in turbulent diffusion
\cite{Majda}.

Foondun and Khoshnevisan \cite{FoondunKh} have shown that the solution
to (\ref{SHE}) is
fairly generically intermittent even when $\sigma$ is nonlinear, as
long as
$\sigma$ behaves as a line in one form or another.

It was noticed early on, in NMR spectroscopy,
that intermittency can be associated strongly to nonlinear noise excitation.
See, for example, Bl\"umich \cite{Blumich}; Lindner et al. \cite{LindnerEtAl}
contains a survey of many related ideas in the physics literature.
In the present context, this informal observation
is equivalent to the existence of
a nonlinear relationship between the energy $\llVert u_t\rrVert
_{L^2(G)}$ of
the solution at time $t$ and the level $\lambda$ of the noise.
A precise form of such a relationship will follow as a ready
consequence of
our present work in all cases where the solution is known (and/or expected)
to be intermittent. In fact, the main findings of this paper will
imply that typically, when the solution
is intermittent, there is a near-dichotomy:
\begin{itemize}
\item On one hand, if $G$ is discrete then the energy of the solution
behaves roughly as $\exp\{\operatorname{const}\cdot\,\lambda^2\}$;
\item on the other hand, if $G$ has a connected locally compact Hausdorff
Abelian subgroup, then
the said energy behaves at least as badly as $\exp\{\operatorname
{const}\cdot\,\lambda^4\}$.
\end{itemize}
And quite remarkably, these properties do not depend
in an essential way on the operator $\mathscr{L}$; they depend
only on the connectivity properties of the underlying state space $G$.

Every standard numerical method for solving (\ref{SHE}) that is known to us
begins by
first discretizing $G$ and $\mathscr{L}$. Our results suggest that
when $\lambda$
is modestly large, then
nearly all such methods will generically underestimate by a vast margin
when we use them to predict the size
of the biggest intermittency islands (or shocks, or spikes)
of the solution to (\ref{SHE}).

Other SPDE models are analyzed in a companion paper
\cite{KhKim} which should ideally be read before the present paper.
That paper is less abstract than this one and, as such,
has fewer mathematical prerequisites.
We present in that paper the
surprising result that the stochastic heat equation on an interval is
typically significantly more noise excitable than the stochastic wave equation
on the real line.

%
%
%re1.1 #&#
\begin{remark}
The referees of the paper have unanimously suggested that
we describe, in words, an intuitive explanation for
this near dichotomy. We agree that such an exposition will add value to
the presentation of
the paper, and would like to say a few things in this direction.
Therefore, let us briefly consider the case that $G$ is a very nice LCA group
(such as a finite group, $\mathbf{Z}^d$, or $\R$)
and $\sigma(u)=c u$ for some constant $c>0$ (the parabolic Anderson model).
First, one can see that when $G$ is finite, (\ref{SHE}) is
another way to write a finite-dimensional stochastic differential equation;
see Examples~\ref{ex1} and~\ref{ex2}. In this
case, it is not hard to verify directly, using only SDE technology,
that the
energy of the solution to (\ref{SHE}) typically grows as $\exp\{
\operatorname{const}\cdot\,\lambda^2\}$
as $\lambda\to\infty$.\footnote{For an example, the reader is
encouraged to consider
the exponential martingale of Brownian motion. In that case, the $\exp
\{\operatorname{const}
\cdot\,\lambda^2\}$ behavior of the solution is more or less immediate.}
In some sense, $G=\mathbf{Z}^d$ can be thought of as a limit
of the finite case: since most of the mass of the solution $u_t$
is concentrated on compacts [because $u_t\in L^2(G)$], this suggests that
the case that $G=\mathbf{Z}^d$ should behave as does the finite case.
And it does.
On the other hand, when $\mathscr{L}$ is the generator of a nice L\'evy
process---say an isotropic $\alpha$-stable process---on $G=\R$, then
$\alpha$ is necessarily in $(1,2]$ (see Dalang \cite{Dalang}),
and a simple scaling argument shows that the large-$\lambda$
behavior of the solution to (\ref{SHE}) is the same as the large-time behavior
of the solution to (\ref{SHE}) with $\lambda=1$, provided that we rescale
time as $T:= \lambda^{2\alpha/(\alpha-1)}t$. The existing
literature on the parabolic Anderson model suggests that the energy
at large time $T$ of the solution to (\ref{SHE}) with $\lambda=1$ should behave
as $\exp\{\operatorname{const}\cdot\, T\}$. Set $T=\lambda^{2\alpha
/(\alpha-1)}t$
in order to see that the energy to (\ref{SHE}) with variable $\lambda\gg1$
ought to behave as $\exp\{\operatorname{const}\cdot\,\lambda^{2\alpha
/(\alpha-1)}t\}$
as $\lambda\to\infty$, for all $t>0$ fixed. In other words, when
$\sigma(u)=cu$
and the underlying L\'evy process is isotropic stable, the energy
behaves as
$\exp\{\operatorname{const}\cdot\,\lambda^q\}$ as $\lambda\to\infty$
for $q=2\alpha/(\alpha-1)\ge4$, where the time variable $t$ is fixed.
\end{remark}

%s2 #&#
\section{Main results}
The main goal of this article is to describe the
behavior of~(\ref{SHE}) for a locally compact Hausdorff Abelian group $G$,
where the initial value $u_0$ is nonrandom and is in the group algebra
$L^2(G)$.\footnote{This is the usual space of all measurable functions
$f\dvtx G\to\R$ that are square integrable with respect to the Haar measure
on $G$.} Compelling, as well as easy to understand, examples
can be found in Section~\ref{secexamples} below.

We assume throughout that the operator $\mathscr{L}$ acts on the space variable
only and denotes the generator of a L\'evy process $X:=\{X_t\}_{t\ge
0}$ on $G$
(see Section~\ref{secLCA} for analysis on LCA groups and Section~\ref{secLevy}
for L\'evy processes on LCA groups), %
$\sigma\dvtx \R\to\R$ is Lipschitz continuous and nonrandom
and $\xi$ denotes space--time white noise on $(0,\infty)\times G$.
That is,
$\xi$ is a generalized centered Gaussian process that is indexed by
$(0,\infty)\times G$ and whose covariance measure is described via
%
%
%e2.1 #&#
\begin{equation}
\operatorname{Cov} \biggl( \int\varphi\,\d\xi, \int\psi\,\d\xi \biggr) =\int
_0^\infty\d t\int_G
m_G(\d x) \varphi_t(x)\psi_t(x),
\end{equation}
for all $\varphi,\psi\in L^2(\d t\times\d m_G)$,
where $m_G$ denotes the Haar measure on $G$, and $\int\varphi\,\d\xi$
and $\int\psi\,\d\xi$ are defined as Wiener integrals. Finally,
$\lambda>0$ designates a fixed parameter that is generally referred
to as the \emph{level of the noise}.

One can adapt the method of Dalang \cite{Dalang} in order to show that,
in the linear case---that is, when
$\sigma\equiv\operatorname{constant}$---(\ref{SHE}) has a function solution
if
{\renewcommand{\theequation}{D}
%e2.2 #&#
\begin{equation}\label{D}
\int_{G^*} \biggl(\frac{1}{\beta+\Re\Psi(\chi)} \biggr) m_{G^*}(
\d\chi)<\infty\qquad\mbox{for one, hence all, }\beta>0,
\end{equation}}%
where $\Psi$ denotes the characteristic exponent of our L\'evy process
$\{X_t\}_{t\ge0}$
and $m_{G^*}$ denotes the Haar measure on the dual $G^*$ to our group $G$.
See also Brze{\'z}niak and Jan van Neerven \cite{BvN} and
Peszat and Zabczyk \cite{PZ}.
Because we want (\ref{SHE}) to have a function solution, at the very least
in the linear case, we have no choice but to assume Dalang's condition
(\ref{D}) from now on.
Henceforth, we assume (\ref{D}) without further mention.

\setcounter{equation}{1}

In some cases, condition (\ref{D}) always holds. For example, suppose $G$ is
discrete. Because $G^*$ is compact, thanks to Pontryagin--van
Kampen duality \cite{Morris,Rudin},
continuity of the function $\Psi$ implies its uniform boundedness, whence
we find that the \emph{Dalang condition} (\ref{D}) \emph{always holds when $G$ is
discrete}. This simple
observation is characteristic of many interesting results about the
heat equation (\ref{SHE})
in the sense that a purely topological property of the group $G$ governs
important aspects of (\ref{SHE}): in this case, we deduce the existence of a
solution generically
when $G$ is discrete. For a probabilistic proof of this particular
fact, see Lemma~\ref{lemexistG-discrete} below.

We wish to establish that ``noise excitation'' properties of (\ref{SHE})
are ``intrinsic to the group $G$.'' This\vspace*{1pt} goal forces us to try
and produce solutions that take values in the group algebra
$L^2(G)$. The following summarizes the resulting existence and
regularity theorem
that is good enough to help us begin our discussion of noise excitation.
We note that an exact definition of a mild solution will be
given in~(\ref{eqmildsoln}). That definition will imply
that our solution is in $L^2(G)$ at all times, and hence is a bona fide
function on $G$
at all times.

%
%
%th2.1 #&#
\begin{theorem}\label{thexistunique}
Suppose that $\sigma$ is Lipschitz continuous and, in addition, that either
$G$ is compact or $\sigma(0)=0$. Then for every nonrandom initial
value $u_0\in L^2(G)$
and $\lambda>0$, the stochastic heat equation (\ref{SHE})
has a mild solution $\{u_t\}_{t\ge0}$, with values in $L^2(G)$,
that satisfies the following:
there exist finite constants $c_1> 0$ and $c_2>0$ that yield the energy
inequality
%
%
%e2.2 #&#
\begin{equation}
\label{eqL2bdd} \E \bigl(\llVert u_t\rrVert_{L^2(G)}^2
\bigr) \le c_1{\e}^{c_2t} \qquad\mbox{for every }t\ge0.
\end{equation}
Moreover, if $v$ is an arbitrary mild solution that satisfies (\ref{eqL2bdd})
subject to $v_0=u_0$,
then $\P\{\llVert u_t-v_t\rrVert _{L^2(G)}=0\}=1$ for all $t\ge0$.
\end{theorem}

%
%
%re2.2 #&#
\begin{remark}
For more explicit bounds on the constants $c_1$ and $c_2$, see
the inequality (\ref{eqNu}) below. That inequality
describes carefully how $c_1$ and $c_2$ depend on the various
parameters of (\ref{SHE})---in particular it states abstractly how
$c_1$ and $c_2$ depend on $\lambda$---and
will be used several times in the sequel.
\end{remark}

The proof of Theorem~\ref{thexistunique} will be given in
Sections~\ref{secproofpart1} and~\ref{secproofpart2};
see also Section~\ref{secconvolutions}, in which we develop the requisite
machinery for Theorem~\ref{thexistunique} and the other main results
in this paper.
However, the preceding result is well known for many Euclidean examples;
see, in particular, Dalang and Mueller \cite{DalangMueller}.

Thus, we assume from now on, and without further mention, that
%
%
%e2.3 #&#
\begin{equation}
\mbox{either }G\mbox{ is compact, or }\sigma(0)=0,
\end{equation}
in order to know {a priori} that (\ref{SHE}) has an $L^2(G)$-valued
solution.\footnote{%
In other words, we do not need to assume that
$\sigma(0)=0$ when $G$ is compact.
However, we do need this condition in general when $G$ is noncompact.
There are examples of $\sigma$ such that $\sigma(0)\neq0$,
noncompact LCA groups $G$,
and L\'evy process generators $\mathscr{L}$ for which (\ref{SHE}) does
not have an $L^2(G)$-valued solution for all time.}

The principal aim of this paper is to study the energy of the solution
when $\lambda$ is large. In order to simplify the exposition, let us
denote the \emph{energy of the solution at time} $t$ by
%
%
%e2.4 #&#
\begin{equation}
\label{energy} \mathscr{E}_t(\lambda):= \sqrt{\E \bigl(
\llVert u_t\rrVert_{L^2(G)}^2 \bigr)}.
\end{equation}
To be more precise, $\mathscr{E}_t(\lambda)$ denotes the $L^2(\P
)$-norm of the
energy of the solution. But we refer to it as the energy in order to
save on
the typography.

We begin our analysis of noise excitation by first noting the following fact:
if~$\sigma$ is essentially bounded and $G$ is compact, then the
solution to (\ref{SHE}) is
\emph{at most linearly noise excitable}. The following is the precise
formulation
of this statement (see Section~\ref{secproofprop} for the proof).

%
%
%pr2.3 #&#
\begin{proposition}[(Linear noise excitation)]\label{prlinear}
If $\sigma\in L^\infty(\R)$ and $G$ is compact, then
%
%
%e2.5 #&#
\begin{equation}
\limsup_{\lambda\uparrow\infty}\frac{\mathscr{E}_t(\lambda)
}{\lambda} <\infty\qquad\mbox{for all
}t>0.
\end{equation}
This bound can be reversed in the following sense:
if also $\inf_{x\in G}\llvert u_0(x)\rrvert>0$ and $\inf_{z\in\R
}\llvert\sigma
(z)\rrvert>0$, then
%
%
%e2.6 #&#
\begin{equation}
\liminf_{\lambda\uparrow\infty}\frac{\mathscr{E}_t(\lambda
)}{\lambda}>0 \qquad\mbox{for all }t>0.
\end{equation}
\end{proposition}

We do not know what happens, at this level of generality, when $\sigma
\in L^\infty(\R)$
and $G$ is noncompact.

The bulk of this paper is concerned with the behavior of (\ref{SHE}) when the
energy $\mathscr{E}_t(\lambda)$ behaves as $\exp(\operatorname{const}\cdot\,
\lambda^q)$, for a fixed positive constant $q$, as $\lambda\uparrow
\infty$.
With this in mind, let us define for all $t>0$,
%
%
%e2.7 #&#
\begin{eqnarray}
\underline{\mathfrak{e}}(t) &:=& \liminf_{\lambda\uparrow\infty}
\frac{\log\log\mathscr{E}_t(\lambda)}{\log\lambda},\qquad\overline {\mathfrak{e}}(t):= \limsup_{\lambda\uparrow\infty}
\frac{\log\log\mathscr{E}_t(\lambda)}{\log
\lambda}.
\end{eqnarray}
If $\underline{\mathfrak{e}}(t)>0$ for all $t>0$, then the solution
to (\ref{SHE})
is expected to be also ``intermittent,'' not only in the usual mathematical
sense \cite{CM94}, but also in a physical sense [i.e., in cases where
the solution to (\ref{SHE}) represents the density of a particle system].

%
%
%de2.4 #&#
\begin{definition}
We refer to $\overline{\mathfrak{e}}(t)$ and
$\underline{\mathfrak{e}}(t)$, respectively, as the
\emph{upper} and the \emph{lower excitation indices}
of $u$ at time $t$.
In many cases of interest, $\underline{\mathfrak{e}}(t)$ and
$\overline{\mathfrak{e}}(t)$ are equal and do not depend
on the time variable $t>0$ (N.B. \emph{not} to be confused with $t\ge0$).
In such cases, we tacitly write $\mathfrak{e}$ for that common value,
and we think of $\mathfrak{e}$ as the
\emph{index of nonlinear noise excitation} of the solution to (\ref{SHE}).
\end{definition}

Thus, Proposition~\ref{prlinear} implies that $\mathfrak{e}=0$
when $\sigma$ is
essentially bounded and $G$ is compact.

As a central part of our analysis, we will prove
that both of these indices are natural quantities, as they are ``group
invariants''
in a sense that will be made clear in Section~\ref{secgroup}.
Moreover, one can deduce from
our work that when $G$ is unimodular
(see Definition~\ref{defmodulus}) the law of the solution to (\ref{SHE}) is
itself a ``group invariant.'' A careful explanation of the quoted terms
will appear
later on in Theorem~\ref{thgroupinv}.
For now, we content ourselves by stating the main three results of this paper.

%
%
%th2.5 #&#
\begin{theorem}[(Discrete case)]\label{thdiscrete}
If $G$ is discrete, then $\overline{\mathfrak{e}}(t)\le2$ for all $t>0$.
In fact, $\mathfrak{e}= 2$,
provided additionally that
%
%
%e2.8 #&#
\begin{equation}
\label{eqLsigma} \ell_\sigma:= \inf_{z\in\R\setminus\{0\}} \bigl\llvert
\sigma(z)/z \bigr\rrvert>0.
\end{equation}
\end{theorem}

Recall that the nonlinearity $\sigma\dvtx \R\to\R$ is assumed to be
Lipschitz continuous,
and hence $\sup_{z\in\R\setminus\{0\}}\llvert\sigma(z)/z\rrvert<\infty$.
Thus, (\ref{eqLsigma}) is the assertion that the graph of $\sigma$
lies globally in some cone.

%
%
%th2.6 #&#
\begin{theorem}[(Connected case)]\label{thconnected1}
Suppose that $G$ is connected
and (\ref{eqLsigma}) holds.
Then $\underline{\mathfrak{e}}(t)\ge4$ for all $t>0$, provided that
in addition either $G$ is noncompact or $G$ is compact, metrizable and
has more than one element.
\end{theorem}

%
%
%re2.7 #&#
\begin{remark}
The proofs will show a slightly more general statement,
thanks to projection. Namely (see Proposition~\ref{prproj1})
that if $G$ contains a noncompact
connected LCA subgroup, or if $G$ contains a compact metrizable connected
LCA subgroup of more than one element,
then $\underline{\mathfrak{e}}(t)\ge4$ as long
as (\ref{eqLsigma}) holds.
\end{remark}

%
%
%th2.8 #&#
\begin{theorem}[(Connected case)]\label{thconnected2}
For every $\theta\ge4$, there are models of the triple
$(G,\mathscr{L},u_0)$ for which $\mathfrak{e}=\theta$.
\end{theorem}

The proofs of the above theorems are presented in Section~\ref{secproofs} below,
and use the results in Sections~\ref{secprojections} and~\ref{seclowerbound}.
In particular, Section~\ref{secprojections} enables us to obtain the
lower bound of the lower excitation index in Theorem~\ref
{thconnected1} ``by projection.''

We now see that if (\ref{eqLsigma}) holds, in addition to the
preceding conditions,
then Theorems~\ref{thdiscrete},~\ref{thconnected1}
and~\ref{thconnected2} together imply the following:
either the energy of the solution behaves
as $\exp(\operatorname{const}\cdot\,\lambda^2)$ or it is greater than\break
$\exp(\operatorname{const}\cdot\,\lambda^4)$ for large noise levels,
and this lower bound cannot be improved upon in general. Moreover,
the connectivity properties of $G$---and not the operator
$\mathscr{L}$---alone determine the first-order strength of the growth of
the energy, viewed as a function of the noise level $\lambda$.

Finally, we will soon see that
when the energy behaves as $\exp(\operatorname{const}\cdot\,\lambda^2)$,
this means that (\ref{SHE}) is only as noise excitable as a classical It\^o stochastic
differential equation. Martin Hairer has asked (private communication) whether
intermittency properties of (\ref{SHE}) are always related to those of the McKean
exponential martingale for Brownian motion. A glance at Example
\ref{ex1} below shows in some sense that,
as far as nonlinear noise excitation is concerned,
\textit{intermittent examples of}~(\ref{SHE}) \textit{behave
as the exponential martingale if and only if $G$ is essentially discrete}.

Throughout, $\lip$ designates the optimal Lipschitz constant of
the function $\sigma$. In more succinct terms, we have
%
%
%e2.9 #&#
\begin{equation}
\label{sigma} \lip:=\sup_{-\infty<x<y<\infty} \biggl\llvert\frac{\sigma
(x)-\sigma(y)}{x-y}
\biggr\rrvert<\infty.
\end{equation}

%s3 #&#
\section{Analysis on LCA groups}\label{secLCA}

We follow the usual terminology of the literature
and refer to a \emph{locally compact Hausdorff Abelian group} as an
\emph{LCA group}. Morris \cite{Morris} and Rudin \cite{Rudin} are
two standard references for the theory of LCA groups.

If $G$ is an LCA group, then we let
$m_G$ denote the Haar measure on $G$.\footnote{That is, $m_G$ is a nonzero
Radon measure on $G$ that is translation invariant under group multiplication.}
The dual, or character, group to $G$ denoted by $G^*$.\footnote{%
That is, $\chi\in G^*$ if and only if $\chi\dvtx G\to\mathbf{C}$ is a
group homomorphism from $G$ to the circle group; that is, $\chi$ is
homeomorphic
and satisfies $\chi(xy)=\chi(x)\chi(y)$ for all $x,y\in G$.
Every $\chi\in G^*$ is called a \emph{character} on $G$. Thus,
for instance, if $G=\R^d$, then $G^*=\R^d$ and $\chi(x)=\exp
(ix\cdot\chi)$.
Also, when $G=\mathbf{Z}^d$, then $G^*=[0,2\pi)^d$ and $\chi
(x)=\exp(ix\cdot\chi)$.}
In addition, the \emph{Fourier
transform} on $L^1(G)$ is defined via the following normalization:
%
%
%e3.1 #&#
\begin{equation}
\qquad \hat{f}(\chi):= \int_G (x,\chi) f(x) m_G( \d
x)\qquad\mbox{for all }\chi\in G^*\mbox{ and }f\in L^1(G),
\end{equation}
where $(x,\chi):=\chi(x):=x(\chi)$ are interchangeable notations that
all describe the natural pairing between $x\in G$ and $\chi\in
G^*$.\footnote{%
This notation is justified by the Pontryagin--van Kampen duality
theorem \cite{Morris,Rudin}:
the dual of $G^*$ is $G$. Consequently, $x\in G$ acts on $\chi\in G^*$
in the same way
as $\chi\in G^*$ acts on $x\in G$, whence $x(\chi)$ can be identified with
$\chi(x)$, as asserted. We emphasize that
different authors use slightly different normalizations of Fourier transforms
from us; see, for example, Rudin \cite{Rudin}.}

Of course, $m_G$ is defined uniquely only up to a multiplicative factor.
Therefore, we always assume the \emph{standard normalization} of
Haar measures; that is any normalization that ensures that
the Fourier transform has a continuous isometric extension to $L^2(G)
= L^2(G^*)$. Analytically speaking, this means that our normalization
of Haar measure ensures that the following formulation of the Plancherel
identity is valid:
%
%
%e3.2 #&#
\begin{equation}
\label{Plancherel} \llVert f\rrVert_{L^2(G)} = \llVert\hat{f}\rrVert
_{L^2(G^*)}\qquad\mbox{for all }f\in L^2(G).
\end{equation}

Our normalization of Haar measure translates to well-known
normalizations of Haar measures via Pontryagin--van Kampen duality
\cite{Morris,Rudin}:
\begin{longlist}[\textit{Case} 1.]
\item[\textit{Case} 1.] If $G$ is compact, then $G^*$ is discrete;
$m_G(G)=1$;
and $m_{G^*}$ denotes the counting measure on subsets of $\Gamma^*$.
\item[\textit{Case} 2.] If $G$ is discrete, then $G^*$ is compact,
$m_{G^*}(G^*)=1$,
and $m_G$ coincides with the counting measure on $G$.
\item[\textit{Case} 3.] If $G=\R^n$ for some integer $n\ge1$,
then $G^*=\R^n$; we may choose $m_G$ and
$m_{G^*}$, in terms of
$n$-dimensional Lebesgue measure, as $m_G(\d x)=a \,\d x$
and $m_{G^*}(\d x)=b \,\d x$ for any two positive reals $a$ and $b$
that satisfy the relation $ab=(2\pi)^{-n}$.
\end{longlist}

%s4 #&#
\section{Some examples}\label{secexamples}

The stochastic PDEs introduced here are quite natural;
in many cases, they are in fact well-established equations.
In this section, we identify some examples to highlight the preceding
claims. Of course, one can begin with the most obvious examples
of stochastic PDEs; for instance, where $G=\R$, $\mathscr{L}=\Delta
$, etc.
But we prefer to have a different viewpoint: as far as interesting examples
are concerned, it is helpful to sometimes think about concrete examples
of LCA groups $G$; then try to understand the L\'evy processes on $G$
(a kind of L\'evy--Khintchine formula) in order to know which
operators $\mathscr{L}$ are relevant. And only then one can think about
the actual resulting stochastic partial differential equation. This
slightly-different
viewpoint produces interesting examples.

%
%
%ex4.1 #&#
\begin{example}[(The trivial group)]\label{ex1}
For our first example, let us
consider the trivial group $G$ with only one element $g$. The
only L\'evy process on this group is $X_t:=g$. All functions on the
group $G$ are, by default, constants. Therefore, $\mathscr{L}f=0$
for all $f\dvtx G\to\R$, and hence $U_t:= u_t(g)$ solves the It\^o SDE
%
%
%e4.1 #&#
\begin{equation}
\d U_t = \lambda\sigma(U_t) \,\d B_t\qquad
\mbox{with }U_0=u_0(g),
\end{equation}
where\vspace*{2pt} $B_t:=\int_{[0,t]\times G} \d\xi$ defines a Brownian motion.
In other
words, when $G$ is the trivial group, (\ref{SHE}) characterizes all drift-free
one-dimensional It\^o diffusions.
\end{example}

%
%
%ex4.2 #&#
\begin{example}[(Cyclic groups, part~I)]\label{ex2}
For a slightly more interesting example consider the cyclic group
$G:=\mathbf{Z}_2$ on two elements. We may think of $G$ as $\mathbf
{Z}/2\mathbf{Z}$; that is,
the set $\{0,1\}$ endowed with binary addition (addition mod 1)
and discrete topology. It is an elementary fact that
the group $G$ admits only one 1-parameter family
of L\'evy processes. Indeed, we can apply the strong Markov property to
the first jump time of $X$ to see that if $X$ is a L\'evy process
on $\mathbf{Z}_2$, then there necessarily exists a number $\kappa\ge
0$ such that,
at independent exponential times,
the process $X$ changes its state at rate $\kappa$:
from 0 to 1 if $X$ is at 0 at the jump time, and from 1 to 0 when $X$
is at 1
at the jump time ($\kappa=0$ yields the constant process).
In this way, we find that (\ref{SHE}) is an encoding of the
coupled two-dimensional SDE
%
%
%e4.2 #&#
\begin{eqnarray}
\d u_t(0) &=& \kappa \bigl[ u_t(1)-u_t(0)
\bigr]\,\d t + \lambda\sigma \bigl(u_t(0) \bigr) \,\d
B_t(0),
\nonumber\\[-8pt]\\[-8pt]\nonumber
\d u_t(1) &=& \kappa \bigl[ u_t(0)-u_t(1)
\bigr]\,\d t + \lambda\sigma \bigl(u_t(1) \bigr) \,\d
B_t(1),
\end{eqnarray}
where $B(0)$ and $B(1)$ are two independent one-dimensional Brownian motions.
In other words, when $G=\mathbf{Z}_2$, (\ref{SHE}) describes a two-dimensional
It\^o diffusion with local diffusion coefficients where
the particles (coordinate processes)
feel an attractive linear drift toward their neighbors (unless $\kappa=0$,
which corresponds to two decoupled diffusions).
\end{example}

%
%
%ex4.3 #&#
\begin{example}[(Cyclic groups, part II)]\label{ex3}
Let us consider the case that $G:=\mathbf{Z}_n$ is the cyclic group on $n$
elements when $n\ge3$. We may think of $G$ as $\mathbf{Z}/n\mathbf
{Z}$; that is,
the set $\{0,\ldots,n-1\}$ endowed with addition $(\mod n)$
and discrete topology.
If $X$ is a L\'evy process on $G$, then
it is easy to see that there exist $n-1$ parameters
$\kappa_1,\ldots,\kappa_{n-1}\ge0$ such that
$X$ jumps (at i.i.d. exponential times) from $i\in\mathbf{Z}/n\mathbf{Z}$
to $i+j(\mod n)$ at rate $\kappa_j$ for every $i\in\{0,\ldots,n-1\}$
and $j\in\{1,\ldots,n-1\}$.
In this case, our stochastic heat equation (\ref{SHE}) is another way to
describe the evolution of the $n$-dimensional It\^o diffusion
$(u(1),\ldots, u(n))$, where for all $i=0,\ldots,n-1$,
%
%
%e4.3 #&#
\begin{equation}
\d u_t(i) = \sum_{j=1}^{n-1}
\kappa_j \bigl[ u_t\bigl(i+j (\mod n)\bigr)-u_t(i)
\bigr] \,\d t +\lambda\sigma \bigl(u_t(i) \bigr) \,\d
B_t(i), \label{eqZn}
\end{equation}
for an independent system $B(0),\ldots,B(n-1)$ of one-dimensional
Brownian motions.
Thus, in this example, (\ref{SHE}) encodes all possible $n$-dimensional
diffusions with
local diffusion coefficients and Ornstein--Uhlenbeck type attractive drifts.
Perhaps the most familiar example of this type is
the simple symmetric case in which $\kappa_1=\kappa_{n-1}:=\kappa>0$
and $\kappa_j=0$ for $j\notin\{1,n-1\}$. In that case,
(\ref{eqZn}) simplifies to
%
%
%e4.4 #&#
\begin{equation}
\d u_t(i) = \kappa(\Delta u_t) (i) + \lambda\sigma
\bigl(u_t(i) \bigr) \,\d B_t(i),
\end{equation}
where $(\Delta f)(i):= f(i\boxplus1)+f(i\boxminus1)-2f(i)$
denotes the ``group Laplacian'' of $f\dvtx \mathbf{Z}_n\to\R$,
$a\boxplus b:=a+b(\mod n)$, and $a\boxminus b:=a-b(\mod n)$.
\end{example}

%
%
%ex4.4 #&#
\begin{example}[(Lattice groups)]\label{ex4}
In this example, $G$ denotes a lattice subgroup of $\R^d$.
This basically means that $G=\delta\mathbf{Z}^d$ for some $\delta>0$
and $d=1,2,\ldots.$ The class of all L\'evy processes on $G$
coincides with
the class of all continuous-time random walks on $G$. Thus, standard
random walk
theory tells us that there exists a constant $\kappa\ge0$---the\vspace*{1pt}
rate---and a probability
function $\{J(y)\}_{y\in\delta\mathbf{Z}^d}$---the so-called jump
measure---such that
$(\mathscr{L}f)(x)=\kappa\sum_{y\in\delta\mathbf{Z}^d}
\{ f(y) - f(x)\} J(y)$, and hence (\ref{SHE}) is an encoding of the following
infinite system of interacting It\^o-type stochastic differential equations:
%
%
%e4.5 #&#
\begin{equation}
\d u_t(x) = \kappa\sum_{y\in\delta\mathbf{Z}^d} \bigl[
u_t(y)-u_t(x) \bigr] J(y) + \lambda\sigma
\bigl(u_t(x) \bigr) \,\d B_t(x),
\end{equation}
for i.i.d. one-dimensional Brownian motions $\{B(z)\}_{z\in\delta
\mathbf{Z}^d}$
and all $x\in\delta\mathbf{Z}^d$. A particularly well-known case is when
$J(y)$ puts equal mass on the neighbors of the origin in $\delta
\mathbf{Z}^d$. In that
case,
%
%
%e4.6 #&#
\begin{equation}
\d u_t(x) = \frac{\kappa}{2d}(\Delta u_t) (x) +
\lambda\sigma \bigl(u_t(x) \bigr) \,\d B_t(x),
\end{equation}
where $(\Delta f)(x):= \sum_{|y-x|= 1} \{f(y)-f(x)\}$ denotes the graph
Laplacian of $f\dvtx \delta\mathbf{Z}^d\to\R$ with $\llvert y-x\rrvert:=\sum_{i=1}^d\llvert y_i-x_i\rrvert$.
\end{example}

%
%
%ex4.5 #&#
\begin{example}[(The real line)]\label{ex5}
As an example, let us choose $G:=\R$ and $X:={}$one-dimensional
Brownian motion
on $\R$. Then $\mathscr{L}f=f''$ and (\ref{SHE}) becomes the usual
stochastic heat equation
%
%
%e4.7 #&#
\begin{equation}
\frac{\partial u_t(x)}{\partial t} = \kappa\frac{\partial^2
u_t(x) }{\partial x^2} + \lambda\sigma
\bigl(u_t(x) \bigr)\xi,
\end{equation}
driven by space--time white noise on $(0,\infty)\times\R$.
\end{example}

%
%
%ex4.6 #&#
\begin{example}[(The torus)]\label{ex6}
Next, we may consider $G:=[0,1)$; as usual we identify the ends of
$[0,1)$
in order to obtain the torus $G:=\mathbf{T}$, endowed with addition
mod 1.
Let $X:={}$Brownian motion on $\mathbf{T}$. Its generator is easily
seen to be
the Laplacian on $[0,1)$ with periodic boundary conditions. Hence, (\ref{SHE})
encodes
%
%
%e4.8 #&#
\begin{equation}
\lleft[ %
\begin{array} {l} \displaystyle\frac{\partial u_t(x)}{\partial t}=
\kappa \frac
{\partial^2 u_t(x) }{\partial x^2} + \lambda\sigma \bigl(u_t(x) \bigr)\xi
\qquad \mbox{for all }0\le x<1,
\\[8pt]
\mbox{subject to }u_t(0)=u_t(1-),
\end{array}
\rright.
\end{equation}
in this case.
\end{example}

%
%
%ex4.7 #&#
\begin{example}[(Totally disconnected examples)]\label{ex7}
Examples~\ref{ex1} through \ref{ex6} are concerned with
more or less standard SDE/SPDE models. Here, we mention one among
many examples where (\ref{SHE}) is more exotic. Consider
$G:=\mathbf{Z}_2\times\mathbf{Z}_2\times\cdots$ to be a countable
direct product of
the cyclic group on two elements. Then $G$ is a compact Abelian group;
this is a group that acts transitively on binary trees and is related
to problems
in fractal percolation. A L\'evy process\vspace*{2pt} on $G$ is simply
a process that has the form $X_t^1\times X_t^2\times\cdots$ at time
$t\ge0$, where
$X^1\times\cdots\times X^k$ is a L\'evy process on
$\prod_{i=1}^k\mathbf{Z}_2$ for every $k\ge1$ (see Example~\ref{ex1}). It is easy to see then that if $f\dvtx G\to\R$ is a function
that is
constant in every coordinate except for the coordinates in some finite
set $F$, then
the generator of $X$ acts on $f$
as $\prod_{j\in F}\mathscr{L}^jf$, where
$\mathscr{L}^j$ denotes the generator of $X^j$ (see Example~\ref{ex1})
and $\mathscr{AB}$ denotes the compositions of operators $\mathscr{A}$
and $\mathscr{B}$.
The resulting stochastic heat equation (\ref{SHE}) is not the subject of our
analysis here
\textit{per se}. Thus, we mention only in passing that, in this case, (\ref{SHE})
appears to have connections to interacting random walks on
a random environment on a binary tree.
\end{example}

%
%
%ex4.8 #&#
\begin{example}[(Positive multiplicative reals)]
Our next, and last example, requires a slightly longer discussion
than its predecessors.
But we feel that this is an illuminating example, and thus worth the effort.

Let
%
%
%e4.9 #&#
\begin{equation}
h(x):= {\e}^x\qquad(x\in\R).
\end{equation}
The range $G:= h(\R)$ of the function $h$ is the multiplicative
positive reals. Frequently, one writes $G$ as $\R^\times_{>0}$;
this is an LCA group, and $h$ is
an isomorphism between $\R$ and $\R^\times_{>0}$. [There are of
course other
topological isomorphisms from $\R$ to $\R^\times_{>0}$; in fact,
$\R\ni x\mapsto\exp(qx)\in\R^\times_{>0}$ works for every real
number $q\neq0$.]
As~$h$ also maps $G^*$ to $\R^*=\R$ homomorphically as well, it
follows that
the dual of $\R^\times_{>0}$ is $\R$, and that the Fourier transform on
$\R^\times_{>0}$ is none other than the classical Mellin transform.
% Recall that the dual group of $\R$ is $\R$ \cite{Morris,Rudin}.
% The dual group of $\R^\times_{>0}$ is just the image under $h(x)=
%\exp(x)$
% of $\R$. That is, the dual of $\R^\times_{>0}$ is $\R$ as well. The
%duality relation between
% $\R^\times_{>0}$ and itself is the image under $h$ of the one for the
%reals.
% Namely,
% \begin{equation}
% (x,\chi):= \chi(x):= x(\chi):= x^{i\chi}\quad\mbox{for $x\in\R^
%\times_{>0}$
% and $\chi\in\R=(\R^\times_{>0})^*$}.
% \end{equation}
% Thus, we can choose the Haar measure on $\R^\times_{>0}$
% and its dual to be respectively
% the measure $ax^{-1}\,\d x$.
% In particular, the choice is unique once we agree on the normalization
% of the Lebesgue measure on $\R$. In particular, if $\xi$ defines a
%white noise on
% $(0,\infty)\times\R$, then
% \begin{equation}\label{eqxitimes}
% \xi^\times_{>0}\left( [0,T]\times B\right):= \xi([0,T]\times
%h^{-1}(B)),
% \end{equation}
% valid for $T>0$ and Borel sets $B\subset\R^\times_{>0}$ of finite
%measure, defines a
% white noise on $(0,\infty)\times\R^\times_{>0}$.

Since $h(x)={\e}^x$ is a topological isomorphism from $\R$ onto $\R
_{>0}^\times$,
every L\'evy process $X:=\{X_t\}_{t\ge0}$ on $\R_{>0}^\times$ can be
written as $X_t=\exp(Y_t)$, where $Y:=\{Y_t\}_{t\ge0}$ is a L\'evy process
on $\R$. An interesting special case is $Y_t=B_t+\delta t$, where
$B:=\{B_t\}_{t\ge0}$
denotes one-dimensional Brownian motion on $\R$ and $\delta\in\R$
is a parameter. Thus,
%
%
%e4.10 #&#
\begin{equation}
t\mapsto X_t:= {\e}^{B_t+\delta t}
\end{equation}
defines a continuous L\'evy process on $\R_{>0}^\times$. The
best-known example
is the case that $\delta=\nicefrac{-1}2$, in which case $X$ is the exponential
martingale.

An application of It\^o's formula (or an appeal to classical
generator computations) shows that
if $f\in C^\infty(\R)$, then for all $x>0$,
%
%
%e4.11 #&#
\begin{equation}
\qquad \E f(xX_t) = f(x) + \frac{t}{2} x^2f''(x)
+ \frac{t(1+2\delta)}{2}xf'(x) +o(t)\qquad\mbox{as }t\downarrow0.
\end{equation}
Thus, we can summarize the preceding as follows:
the exponential martingale is a L\'evy process on $\R_{>0}^\times$
with generator\vspace*{1pt}
$(\mathscr{L}f)(x) = \frac{1}2 x^2 f''(x) + (\delta+\frac{1}2)xf'(x)$.
Thus, we can understand
our stochastic heat equation (\ref{SHE}), in this context, as the following
Euclidean SPDE:
%
%
%e4.12 #&#
\begin{equation}
\frac{\partial u_t(x)}{\partial t} = \frac{x^2}{2} \frac{\partial^2
u_t(x)}{\partial x^2} + \biggl(\delta+
\frac{1}2 \biggr)x\frac{\partial u_t(x)}{\partial x}+ \lambda\sigma \bigl(u_t(x)
\bigr)\xi_h;
\end{equation}
for $t,x>0$. Moreover, $\xi_h$ denotes a space--time white noise
on $(0,\infty)\times(0,\infty)$ whose control measure is
proportional to $x^{-1}\, \d t \,\d x\,\1_{(0,\infty)^2}(t,x)$
[the restriction of the Haar measure on $\R_+\times{\mathbf
R}_{>0}^\times$ to
$(0,\infty)\times{\mathbf R}_{>0}^\times$]. We expend a few lines
and make the following amusing observation as an aside: from the
perspective of
these SPDEs, the most natural case is the drift-free case where $\delta
=\nicefrac{-1}2$.
In that case, the underlying L\'evy process $X$ is the exponential martingale,
as was noted earlier. The exponential martingale is one of the
archetypal classical
examples of an intermittent process \cite{ZRS}.
Moreover, $X$ is centered when $\delta=\nicefrac{-1}2$
in the sense that $\E X_t$ is the group identity.
Interestingly enough, the exponential martingale is natural
in other sense as well: (1) The process $X$ is a natural candidate for being
a ``Gaussian'' process with values in the group $\R^\times_{>0}$
in the sense that $X$ is the image of a real-valued Gaussian process
under the exponential map; and (2) $X$ has quadratic variation $t$,
that is,
%
%
%e4.13 #&#
\begin{equation}
\quad\lim_{n\to\infty} \sum_{0\le k\le2^nt} \bigl[
X_{(k+1)/2^n} X_{k/2^n}^{-1} \bigr]^2 =t\qquad
\mbox{almost surely for all }t\ge0.
\end{equation}
This property can be verified by standard methods.
\end{example}

%s5 #&#
\section{L\'evy processes} \label{secLevy}
Let us recall some basic facts about L\'evy processes on~LCA groups.
For more details, see Berg and Forst \cite{BergForst}
and Port and Stone \mbox{\cite{PortStonea,PortStoneb}}.
Bertoin \cite{Bertoin} and Jacob \cite{Jacob}
are masterly accounts of the probabilistic and analytic aspects of the
theory of L\'evy
processes on $\R^n$ and $\mathbf{Z}^n$.

Throughout, $(\Omega,\F,\P)$ is a fixed probability space.

Let $G$ denote an LCA group, and suppose
$Y:=\{Y_t\}_{t\ge0}$ is a stochastic process
on $(\Omega,\F,\P)$ with values in $G$.
[We always opt to write $Y_t$ in place of $Y(t)$, as is customary in
the theory of stochastic processes.]
We say that $Y$ is a \emph{L\'evy process} on $G$ if:
\begin{longlist}[(2)]
\item[(1)] $Y_0={\e}_G$, the identity element of $G$;
\item[(2)] $Y_{t+s}Y_s^{-1}$ is independent of $\{Y_u\}_{u\in[0,s]}$
and has the same distribution as $Y_t$, for all $s,t\ge0$; and
\item[(3)] the random function $t\mapsto Y_t$ is right continuous
and has left limits everywhere with probability one.
\end{longlist}

Our definition might appear to be slightly more stringent than
the standard definition, but turns out to be equivalent to the standard
definition,
for instance, when $G$ is metrizable.

Let $\mu_t:=\P\circ Y_t^{-1}$ denote the distribution of the
random variable $Y_t$.
Then $\{P_t\}_{t\ge0}$ is a convolution semigroup, where
%
%
%e5.1 #&#
\begin{equation}
(P_tf) (x):= \E f(x Y_t):= \int_G
f(xy) \mu_t(\d y).
\end{equation}
We can always write the Fourier transform of the probability
measure $\mu_t$ as follows:
%
%
%e5.2 #&#
\begin{equation}
\hat{\mu}_t(\chi) = \E(Y_t, \chi) = {
\e}^{-t\Psi(\chi)} \qquad\mbox{for all }t\ge0\mbox{ and }\chi\in G^*,
\end{equation}
where $\Psi\dvtx G^*\to\mathbf{C}$ is continuous and
$\Psi({\e}_{G^*})=0$.
It is easy to see that Dalang's condition (\ref{D}) always implies the following:
%
%
%e5.3 #&#
\begin{equation}
\label{eqPDF} \int_{G^*} {\e}^{-t\Re\Psi(\chi)}
m_{G^*} (\d\chi)<\infty\qquad\mbox{for all }t>0.
\end{equation}
See, for example, \cite{FKN}, Lemma 8.1.
In this case, the following is well defined:
%
%
%e5.4 #&#
\begin{equation}
\label{eqinversion} p_t(x) = \int_{G^*}
\bigl(x^{-1},\chi \bigr){\e}^{-t\Psi(\chi)} m_{G^*}(\d\chi)
\qquad\mbox{for all }t>0\mbox{ and }x\in G.
\end{equation}

The following is a consequence of Fubini's theorem.

%
%
%le5.1 #&#
\begin{lemma}\label{lemTrDensity}
The function $(t,x)\mapsto p_t(x)$ is well defined and bounded as
well as
uniformly continuous
for $(t,x)\in[\delta,\infty)\times G$ for every fixed $\delta
>0$. Moreover,
we can describe the semigroup via
%
%
%e5.5 #&#
\begin{equation}
\qquad (P_tf) (x) = \int f(xy)p_t(y) m(\d y)\qquad\mbox{for
all }t>0, x\in G, f\in L^1(G).
\end{equation}
Consequently, $p_t(x)\ge0$
for all $t>0$ and $x\in G$.
\end{lemma}

We omit the proof, as it is elementary. Let us mention, however, that
the preceding lemma
guarantees that the Chapman--Kolmogorov equation holds pointwise. That is,
%
%
%e5.6 #&#
\begin{equation}
\label{eqChapman-Kolmogorov} p_{t+s}(x) = (p_t*p_s) (x)
\qquad\mbox{for all }s,t>0\mbox{ and }x\in G,
\end{equation}
where ``$*$'' denotes the usual convolution on $L^1(G)$, that is,
%
%
%e5.7 #&#
\begin{equation}
(f*g) (x):=\int_G f(y)g \bigl(xy^{-1} \bigr)
m_G(\d y).
\end{equation}

Define, for all $t>0$ and $x\in G$,
%
%
%e5.8 #&#
\begin{equation}
\bar{p}_t(x):= (P_tp_t) (x) = \int
_G p_t(xy) p_t(y)
m_G(\d y).
\end{equation}
In particular, we apply the preceding with $x:={\e}_G$ in order to
see that
%
%
%e5.9 #&#
\begin{equation}
\bar{p}_t({\e}_G) = \llVert p_t\rrVert
_{L^2(G)}^2 \qquad\mbox{for all }t>0.
\end{equation}
Furthermore, it can be shown that the following inversion
theorem holds for all $t>0$ and $x\in G$:
%
%
%e5.10 #&#
\begin{equation}
\bar{p}_t(x) = \int_{G^*} \bigl(x^{-1},
\chi \bigr){\e}^{-2t\Re\Psi(\chi
)} m_{G^*}(\d\chi).
\end{equation}
Thus, we find that
%
%
%e5.11 #&#
\begin{equation}
\label{eqUpsilon} \Upsilon(\beta):= \int_0^\infty{
\e}^{-\beta t}\llVert p_t\rrVert_{L^2(G)}^2 \,\d t
\end{equation}
satisfies
%
%
%e5.12 #&#
\begin{equation}
\label{eqUpsilon1} \Upsilon(\beta):= \int_0^\infty{
\e}^{-\beta t}\llVert p_t\rrVert_{L^2(G)}^2 \,\d t = \int_{G^*}\frac{m_{G^*}(\d\chi)}{\beta+2\Re\Psi(\chi)}.
\end{equation}
Consequently, Dalang's condition (\ref{D}) can be recast equivalently and
succinctly as
the condition that $\Upsilon\dvtx [0,\infty)\to[0,\infty]$ is finite
on $(0,\infty)$.

Since $t\mapsto\int_0^t\bar{p}_s({\e}_G)\, \d s$ is nondecreasing,
Lemma 3.3 of \cite{FKN} implies the following Abelian/Tauberian bound:
%
%
%e5.13 #&#
\begin{equation}
\label{eqpUpsilon} {\e}^{-1}\Upsilon(1/t) \le\int_0^t
\bar{p}_s({\e}_G)\, \d s\le{\e}\Upsilon(1/t) \qquad
\mbox{for all }t>0.
\end{equation}
Finally, by \emph{the generator} of $\{X_t\}_{t\ge0}$ we mean
the linear operator $\mathscr{L}$ with domain
%
%
%e5.14 #&#
\begin{equation}
\operatorname{Dom}[\mathscr{L}]:= \Bigl\{ f\in L^2(G)\dvtx
\mathscr{L}f:= \lim_{t\downarrow0} t^{-1}(P_tf-f)
\mbox{ in }L^2(G) \Bigr\}.
\end{equation}
This defines $\mathscr{L}$ as an $L^2$-generator, which is a slightly
different operator than the one that is usually obtained from the
Hille--Yosida theorem.
The $L^2$-theory makes good sense here for a number of reasons; chief among
them is the fact that $G$ need not be second countable, and hence the
standard form
of the Hille--Yosida theorem is not applicable. The $L^2$-theory has
the added advantage
that the domain is more or less explicit, as will be seen shortly.

Recall that each $P_t$ is a contraction on $L^2(G)$, and observe that
%
%
%e5.15 #&#
\begin{equation}
\widehat{P_t f}(\chi) = \hat{f}(\chi) \exp \bigl\{-t \overline{
\Psi(\chi)} \bigr\}\qquad\mbox{for all }t\ge0\mbox{ and }\chi\in G^*.
\end{equation}
Therefore, for all $f,g\in L^2(G)$,
%
%
%e5.16 #&#
\begin{equation}
\int_G g(P_tf-f)\, \d m_G = -\int
_{G^*} \hat{f}(\chi) \overline{\hat{g}(\chi)} \bigl( 1 - {
\e}^{-t\overline{\Psi(\chi)}} \bigr) m_{G^*}(\d\chi).
\end{equation}
It follows fairly readily from this relation that $\mathscr
{L}\dvtx \operatorname{Dom}[\mathscr{L}]
\to L^2(G)$,
%
%
%e5.17 #&#
\begin{equation}
\operatorname{Dom}[\mathscr{L}] = \biggl\{ f\in L^2(G)\dvtx \int
_{G^*} \bigl\llvert\hat{f}(\chi) \bigr\rrvert^2
\bigl\llvert\Psi(\chi) \bigr\rrvert^2 m_{G^*}(\d\chi)< \infty
\biggr\},
\end{equation}
and for all $f\in\operatorname{Dom}[\mathscr{L}]$ and $g\in L^2(G)$,
%
%
%e5.18 #&#
\begin{equation}
\int_G g\mathscr{L}f \,\d m_G = - \int
_{G^*}\hat{f}(\chi) \overline{\hat{g}(\chi) \Psi(\chi)}
m_{G^*}(\d\chi).
\end{equation}
The latter identity is another way to write
%
%
%e5.19 #&#
\begin{equation}
\widehat{\mathscr{L}f}(\chi) = -\hat{f}(\chi)\overline{\Psi(\chi)} \qquad
\mbox{for all }f\in\operatorname{Dom}[\mathscr{L}]\mbox{ and }\chi\in G^*.
\end{equation}
In other words, $\mathscr{L}$ is a pseudo-differential operator on $L^2(G)$
with Fourier multiplier (``symbol'') $-\overline{\Psi}$.

%s6 #&#
\section{Stochastic convolutions}\label{secconvolutions}
Throughout this paper, $\xi$ will denote space--time white noise on
$\R_+\times G$. That is, $\xi$ is a set-indexed Gaussian random field,
indexed by Borel subsets of $\R_+\times G$ that have finite measure
$\operatorname{Leb}\times m_G$
(product of Lebesgue and Haar measures, resp., on $\R_+$ and
$G$). Moreover,
$\E\xi(A\times T)=0$ for all measurable $A\subset\R_+$ and
$T\subset G$
of finite measure (resp., Lebesgue and Haar), and
%
%
%e6.1 #&#
\begin{equation}
\operatorname{Cov} \bigl( \xi(B\times T), \xi(A\times S) \bigr) =
\operatorname{Leb}(B\cap A)\cdot m_G (T\cap S),
\end{equation}
for all Borel sets $A,B\subset\R_+$ that have finite Lebesgue measure
and all Borel sets $S,T\subseteq G$ that have finite Haar measure. It
is easy to see
that $\xi$ is then a vector-valued measure with values in $L^2(\P)$.

The principal goal of this section is to introduce and study stochastic
convolutions of the form
%
%
%e6.2 #&#
\begin{equation}
(K\circledast Z )_t(x):= \int_{(0,t)\times G}
K_{t-s} \bigl(yx^{-1} \bigr) Z_s(y) \xi(\d s \,\d y),
\end{equation}
where $Z$ is a suitable space--time random field and $K$ is a nice nonrandom
space--time function from $(0,\infty)\times G$ to $\R$;
Lemma~\ref{lemYoung} below will make precise the meaning
of ``suitable'' in this context.

If $Z$ is a predictable random field, in the sense of Walsh \cite
{Walsh} and Dalang
\cite{Dalang}, and satisfies
%
%
%e6.3 #&#
\begin{equation}
\quad \sup_{t\in[0,T]}\sup_{x\in G}\E \bigl( \bigl\llvert
Z_t(x) \bigr\rrvert^2 \bigr)<\infty,\qquad \int
_0^T\d s\int_G
m_G(\d y) \bigl[K_s(y) \bigr]^2<\infty,
\end{equation}
for all $T>0$, then the stochastic convolution $K\circledast Z$ is the
same stochastic
integral that has been obtained in Walsh \cite{Walsh} and, in
particular, Dalang \cite{Dalang}.
One of the essential properties of the resulting stochastic integral is the
following $L^2$ isometry:
%
%
%e6.4 #&#
\begin{equation}
\quad\E \bigl( \bigl\llvert(K\circledast Z )_t(x) \bigr\rrvert
^2 \bigr) = \int_0^t\d s\int
_G m_G(\d y) \bigl[ K_{t-s}
\bigl(yx^{-1} \bigr) \bigr]^2\E \bigl( \bigl\llvert
Z_s(y) \bigr\rrvert^2 \bigr).
\end{equation}

In this section, we briefly describe an extension of the Walsh--Dalang
stochastic integral
that has the property that $t\mapsto(K\circledast Z)_t$ is a
stochastic process
with values in the group algebra $L^2(G)$. Thus, the resulting
stochastic convolution
need not be, and in general is not, a random field in the modern sense
of the word.
Rather, we can realize the stochastic convolution
process $t\mapsto(K\circledast Z)_t$ as a Hilbert-space-valued
stochastic process, where the Hilbert space is $L^2(G)$.

Our construction has a similar flavor as some other recent constructions;
see, in particular, Da Prato and Zabczyk \cite{DZ} and
Dalang and Quer--Sardanyons \cite{DalangQuer}. However, our construction
also has some novel aspects.

Let us set forth some notation first.
As always, let $(\Omega,\F,\P)$ denote a probability space.

%
%
%de6.1 #&#
\begin{definition}
Let $Z:=\{Z_t(x)\}_{t\in I,x\in G}$ be a two-parameter (space--time)
real-valued stochastic process
indexed by $I\times G$, where $I$ is a measurable subset of $\R_+$.
We say that $Z$ is a \emph{random field} when
the function $Z\dvtx (\omega,t,x)\mapsto Z_t(x)(\omega)$
is product measurable from $\Omega\times I\times G$ to $\R$.
\end{definition}

The preceding definition is somewhat unconventional; our random fields are
frequently referred to as ``universally measurable random fields.''
Because we
will never have need for any other random fields than universally measurable
ones, we feel justified in abbreviating the terminology.

%
%
%de6.2 #&#
\begin{definition}
For every random field $Z:=\{Z_t(x)\}_{t\ge0,x\in G}$
and $\beta\ge0$, let us define
%
%
%e6.5 #&#
\begin{equation}
\label{N} \mathcal{N}_\beta(Z;G):= \sup_{t\ge0}
\bigl\{ {\e}^{-2\beta t} \E \bigl( \llVert Z_t\rrVert
_{L^2(G)}^2 \bigr) \bigr\}^{1/2}.
\end{equation}
We may sometimes only write $\mathcal{N}_\beta(Z)$ when it is clear
which underlying group we are referring to.
\end{definition}

Each $\mathcal{N}_\beta$ defines a norm on space--time random fields,
provided that we identify a random field with all of its versions.

%
%
%de6.3 #&#
\begin{definition}\label{defL2beta}
For every $\beta\ge0$, we define $\mathcal{L}^2_\beta(G)$ be the $L^2$-space
of all measurable functions $\Phi\dvtx (0,\infty)\times G\to\R$
with $\llVert \Phi\rrVert _{\mathcal{L}^2_\beta(G)}<\infty$, where
%
%
%e6.6 #&#
\begin{equation}
\llVert\Phi\rrVert_{\mathcal{L}^2_\beta(G)}^2:= \int_0^\infty{
\e}^{-2\beta s}\llVert\Phi_s\rrVert_{L^2(G)}^2\,\d s.
\end{equation}
\end{definition}

We emphasize that the elements of $\mathcal{L}^2_\beta(G)$ are
nonrandom.

Define, for every $\varphi\in L^2(G)$ and $t\ge0$,
%
%
%e6.7 #&#
\begin{equation}
B_t(\varphi):= \int_{(0,t)\times G}\varphi(y) \xi(\d s \,\d y).
\end{equation}
The preceding is understood as a Wiener integral, and it is easy to see that
$\{B_t(\varphi)\}_{t\ge0}$ is Brownian motion scaled to have variance
$\llVert \varphi\rrVert _{L^2(G)}$ at time one. Let $\F_t$ denote the
$\sigma$-algebra
generated by all random variables of the form
$B_s(\varphi)$, as $s$ ranges within $[0,t]$
and $\varphi$ ranges within $L^2(G)$. Then $\{\F_t\}_{t\ge0}$ is
the (raw) filtration of the white noise $\xi$. Without changing the
notation, we will complete $[\P]$ every \mbox{$\sigma$-}algebra $\F_t$ and
also make $\{\F_t\}_{t\ge0}$ right continuous in the usual way. In
this way,
we may apply the martingale-measure machinery of Walsh \cite{Walsh}
whenever we need to.

A space--time stochastic process $Z:=\{Z_t(x)\}_{t\ge0,x\in G}$
is called an \emph{elementary random field} \cite{Walsh} if we can write
$Z_t(x) = X \mathbf{1}_{[a,b)}(t)\psi(x)$,
where $0<a<b$, $\psi\in C_c(G)$
(the usual\vspace*{1pt} space of real-valued continuous functions with compact
support on $G$),
and $X\in L^2(\P)$ is $\F_a$-measurable. Clearly, elementary random
fields are random fields in the sense mentioned earlier.

A space--time stochastic process is a \emph{simple random field} \cite{Walsh}
if it is a finite nonrandom sum of elementary random fields.

%
%
%de6.4 #&#
\begin{definition}\label{defP}
For\vspace*{2pt} every $\beta\ge0$,
we define $\mathcal{P}^2_\beta(G)$ to be the completion of the
collection of simple
random fields in the norm $\mathcal{N}_\beta$.
We may observe that: (i)~Every $\mathcal{P}^2_\beta(G)$ is a
Banach space, once endowed
with norm $\mathcal{N}_\beta$; and
(ii)~if $\alpha<\beta$, then $\mathcal{P}^2_\alpha(G)
\subseteq\mathcal{P}^2_\beta(G)$.
\end{definition}

We can think of an element of $\mathcal{P}^2_\beta(G)$ as a
``predictable random field'' in some extended sense.

Let us observe that if $K\in\mathcal{L}^2_\beta(G)$, then
$\int_0^T\d s\int_G m_G(\d y) [K_s(y)]^2<\infty$
all \mbox{$T>0$}. Indeed,
%
%
%e6.8 #&#
\begin{equation}
\int_0^T\d s\int_G
m_G (\d y) \bigl[K_s(y) \bigr]^2 \le{
\e}^{2\beta T}\llVert K\rrVert_{\mathcal{L}^2_\beta
(G)}^2.
\end{equation}
Therefore, we can define the stochastic convolution
$K\circledast Z$ for all simple random fields $Z$
and all $K\in\mathcal{L}^2_\beta(G)$ as in Walsh \cite{Walsh}.
The following\vspace*{2pt} yields further information on this stochastic convolution.
For other versions of such stochastic Young inequalities, see
Foondun and Khoshnevisan \cite{FoondunKh}, and especially
Conus and Khoshnevisan \cite{ConusK}.

%
%
%le6.5 #&#
\begin{lemma}[(Stochastic Young inequality)]\label{lemYoung}
Suppose that $Z$ is a simple random field and $K\in\mathcal
{L}^2_\beta(G)$
for some $\beta\ge0$. Then $K\circledast Z\in\mathcal{P}^2_\beta(G)$,
and
%
%
%e6.9 #&#
\begin{equation}
\label{eqYoung} \mathcal{N}_\beta( K\circledast Z ) \le
\mathcal{N}_\beta(Z)\cdot\llVert K\rrVert_{\mathcal{L}^2_\beta(G)}.
\end{equation}
\end{lemma}

If $K\in\mathcal{L}^2_\beta(G)$, then
Walsh's theory \cite{Walsh} produces a space--time stochastic
process $(t,x)\mapsto(K\circledast Z)_t(x)$; that is, a collection of
random variables $(K\circledast Z)_t(x)$, one for every $(t,x)\in
(0,\infty)\times G$.
Thus, the stochastic convolution in Lemma~\ref{lemYoung} is well defined.

Lemma~\ref{lemYoung} implies that the stochastic convolution operator
$K\circledast\bullet$ is a bounded
linear map from $Z\in\mathcal{P}^2_\beta(G)$ to
$K\circledast Z\in\mathcal{P}^2_\beta(G)$
with operator norm being at most $\llVert K\rrVert _{\mathcal
{L}^2_\beta(G)}$.
In particular, it follows readily from this lemma that $K\circledast Z$ is
a random field, since it is an element of $\mathcal{P}^2_\beta(G)$.

\begin{pf*}{Proof of Lemma \ref{lemYoung}}
It suffices to consider the case that $Z$ is an elementary random field.

Let us say that a function $K\dvtx (0,\infty)\times G\to\R$ is \emph
{elementary}
(in the sense of Lebesgue) if we can write $K_s(y) = A\mathbf{1}_{[c,d)}(s)
\phi(y)$ where $A\in\R$, $0\le c<d$, and $\phi\in C_c(G)$
(the usual space of continuous real-valued functions on $G$ that have
compact support).
Let us say also that $K$ is a \emph{simple}
function (also in the sense of Lebesgue) if it is a finite sum of
elementary functions.
These are small variations on the usual definitions of the Lebesgue theory
of integration. But they produce the same theory as that of Lebesgue.
Here, these variations are particularly handy.

From now on, let us choose and fix some constant
$\beta\ge0$, and let us observe that if $K$ were an elementary
function, then
$K\in\mathcal{L}^2_\beta(G)$ for every $\beta\ge0$.

Suppose we could establish (\ref{eqYoung}) in the case that $K$ is
an elementary function. Then of course (\ref{eqYoung}) also holds
when $K$ is a simple function. Because $C_c(G)$ is dense in $L^1(m_G)$
\cite{Rudin}, E8, page~268,
the usual form of Lebesgue's theory ensures that simple
functions are dense in $\mathcal{L}^2_\beta(G)$.
Therefore, by density, if we could prove that ``$K\circledast Z\in
\mathcal{P}^2_\beta(G)$''
and (\ref{eqYoung}) both hold in the case that $K$ is elementary,
then we can
deduce ``$K\circledast Z\in\mathcal{P}^2_\beta(G)$'' and
(\ref{eqYoung}) for all $K\in\mathcal{L}^2_\beta(G)$.
This reduces our entire problem to the case where $Z$ is an elementary random
field and $K$ is an elementary function, properties that we assume to
be valid
throughout the remainder of this proof. Thus, from now on
we consider
%
%
%e6.10 #&#
\begin{equation}
\label{eqKZ} K_s(y) = A\cdot\mathbf{1}_{[c,d)}(s) \phi(y)
\quad\mbox{and}\quad Z_t(x) = X\cdot\mathbf{1}_{[a,b)}(t)
\psi(x),
\end{equation}
where $A\in\R$, $0\le c<d$, $0<a<b$, $X\in L^2(\P)$ is $\F_a$-measurable,
$\psi\in C_c(G)$, and $\phi\in C_c(G)$.
The remainder of the proof works is divided naturally into three
steps.

\begin{longlist}[\textit{Step} 2.]
\item[\textit{Step} 1 (\textit{measurability}).]
We first show that $K\circledast Z$ is a random field in the sense of
this paper.

Choose and fix some $T>0$.
According to the Walsh theory \cite{Walsh},
%
%
%e6.11 #&#
\begin{equation}
(K\circledast Z)_t(x) = AX\cdot\int_{\mathcal{T}(t)\times G} \phi
\bigl(yx^{-1} \bigr)\psi(y) \xi(\d s \,\d y),
\end{equation}
where $\mathcal{T}(t):= (0,t)\cap[a,b)\cap[t-d,t-c)$,
and the stochastic integral can be understood as a Wiener integral,
since the integrand is nonrandom and square integrable $[\d s\times
m_G(\d y)]$.
In particular, we may observe that for all $x,w\in G$ and $t\in[0,T]$,
%
%
%e6.12 #&#
\begin{eqnarray}
&& \E \bigl( \bigl\llvert(K\circledast Z)_t(x) - (K\circledast
Z)_t(w) \bigr\rrvert^2 \bigr)\nonumber
\\
&&\qquad = A^2\E
\bigl(X^2 \bigr) \bigl\llvert\mathcal{T}(t) \bigr\rrvert\cdot\int
_G m_G(\d y) \bigl[\psi(y)
\bigr]^2 \bigl\llvert\phi \bigl(yx^{-1} \bigr)-\phi
\bigl(yw^{-1} \bigr) \bigr\rrvert^2
\\
&&\qquad \le\operatorname{const}\cdot\int_G \bigl\llvert\phi
\bigl(yw^{-1}x \bigr)-\phi(y) \bigr\rrvert^2
m_G(\d y),\nonumber
\end{eqnarray}
where $\llvert\mathcal{T}(t)\rrvert=t(b-a)(d-c)$ denotes the Lebesgue
measure of
$\mathcal{T}(t)$,
and the implied constant does not depend on $(t,x,w)\in[0,T]\times
G\times G$. Similarly,
for every $0\le t\le\tau\le T$ and $x\in G$,
%
%
%e6.13 #&#
\begin{equation}
\E \bigl( \bigl\llvert(K\circledast Z)_t(x) - (K\circledast
Z)_\tau(x) \bigr\rrvert^2 \bigr) \le\operatorname{const}\cdot\,(\tau-t),
\end{equation}
where the implied constant does not depend on $(t,x, w)\in[0,T]\times
G\times G$.
Consequently,
%
%
%e6.14 #&#
\begin{equation}
\label{eqcontL2} \mathop{\lim_{x\to w}}_{t\to\tau} \E \bigl(
\bigl\llvert( K\circledast Z )_t(x) - ( K\circledast Z
)_\tau(w) \bigr\rrvert^2 \bigr)=0,
\end{equation}
uniformly for all $\tau\in[0,T]$ and $w\in G$.
In light of a separability theorem of Doob~\cite{Doob}, Chapter~2, the
preceding implies that
$(\Omega,(0,\infty),G)\ni
(\omega,t,x)\mapsto(K\circledast Z)_t(x)(\omega)$
has a product-measurable version.\footnote{As written, Doob's theorem is
applicable to the case of stochastic processes that are
indexed by Euclidean spaces. But the very same proof will work for processes
that are indexed by $\R_+\times G$.}
\end{longlist}

\begin{longlist}[\textit{Step} 2.]
\item[\textit{Step} 2 (\textit{extended predictability}).]
Next, we prove that $K\circledast Z\in\mathcal{P}^2_\beta(G)$.

Let\vspace*{1pt} us define another elementary function $\bar{K}_s(y):= A\mathbf
{1}_{[c,d)}(s)
\bar\phi(y)$ where $A$~and~$(c,d)$ are the same as they were in the
construction of $K$, but $\bar\phi\in L^2(G)$ is not necessarily the
same as $\phi$. It is easy to see that
%
%
%e6.15 #&#
\begin{eqnarray}\label{eqprepre}
&& \E \bigl( \bigl\llvert(K\circledast Z)_t(x) - (
\bar{K}\circledast Z)_t(x) \bigr\rrvert^2 \bigr)\nonumber
\\
&&\qquad  =
A^2\E \bigl(X^2 \bigr) \bigl\llvert\mathcal{T}(t) \bigr
\rrvert\cdot\int_G \bigl[\psi(y) \bigr]^2 \bigl
\llvert\phi \bigl(yx^{-1} \bigr)- \bar\phi \bigl(yx^{-1}
\bigr) \bigr\rrvert^2 m_G(\d y)
\\
&&\qquad \le\operatorname{const}\cdot\,\llVert\phi-\bar\phi\rrVert_{L^2(G)}^2,\nonumber
\end{eqnarray}
where the implied constant does not depend on $(t,x,\phi,\bar
\phi)$.
The definition of the stochastic convolution shows that
%
%
%e6.16 #&#
\begin{equation}
\operatorname{supp} \bigl( ( K\circledast Z)_t \bigr) \subseteq
\operatorname{supp}(\psi)\oplus\operatorname{supp}(\phi),
\end{equation}
almost surely for all $t\ge0$, where ``supp'' denotes ``support.''
Since
$K\circledast Z$ and $\bar{K}\circledast Z$ are both random fields
(step~1),
we can integrate both sides of (\ref{eqprepre})
$[\exp(-2\beta t)\, \d t\times m_G(\d x)]$ in order to find that
%
%
%e6.17 #&#
\begin{equation}
\label{eqK-K} \qquad\bigl[ \mathcal{N}_\beta(K\circledast Z - \bar{K}
\circledast Z) \bigr]^2 \le\operatorname{const}\cdot\,\llVert\phi-\bar
\phi\rrVert_{L^2(G)}^2 \cdot m_G \bigl(
\operatorname{supp}(\psi)\oplus S \bigr),
\end{equation}
where $ S $ is any compact set that contains both the supports of
both $\phi$ and $\bar\phi$. Of course, $\operatorname{supp}(\psi
)\oplus S $
has finite $m_G$-measure since it is a compact set.

We now use the preceding computations as follows: let us choose in
place of
$\bar\phi$ a sequence of functions $\phi^1,\phi^2,\ldots,$ all in $L^2(G)$
and all supported in one fixed compact set $ S \supset
\operatorname{supp}(\phi)$, such that: (i) Each $\phi^j$ can be
written as
$\phi^j(x):= \sum_{i=1}^{n_j} a_{i,j}\mathbf{1}_{E_i}(x)$
for some constants $a_{i,j}$'s and compact sets $E_j\subset G$; and (ii)
$\llVert \phi-\phi^j\rrVert _{L^2(G)}\to0$ as $j\to\infty$. The
resulting kernel
can be written as $K^j$ (in place of $\bar K$). Thanks to (\ref{eqK-K}),
%
%
%e6.18 #&#
\begin{equation}
\lim_{j\to\infty}\mathcal{N}_\beta \bigl( K\circledast Z -
K^j\circledast Z \bigr)=0.
\end{equation}
A direct computation shows that $K^j\circledast Z$ is an elementary random
field, and hence it is in $\mathcal{P}^2_\beta$. Thanks\vspace*{1.5pt} to the preceding
display, $K\circledast Z$ is also in $\mathcal{P}^2_\beta$. This
completes the proof of
step~2.
\end{longlist}

\begin{longlist}[\textit{Step} 2.]
\item[\textit{Step} 3 {[\textit{proof of} (\ref{eqYoung})]}.]
Since
%
%
%e6.19 #&#
\begin{equation}
\label{eqSP} \qquad \E \bigl( \bigl\llvert(K\circledast Z )_t(x) \bigr
\rrvert^2 \bigr) = \int_0^t\d s
\int_G m_G(\d y) \bigl[K_{t-s}
\bigl(yx^{-1} \bigr) \bigr]^2\E \bigl( \bigl\llvert
Z_s(y) \bigr\rrvert^2 \bigr),
\end{equation}
we integrate both sides $[\d m]$ in order to obtain
%
%
%e6.20 #&#
\begin{eqnarray}
\E \bigl( \bigl\llVert(K\circledast Z )_t \bigr\rrVert
_{L^2(G)}^2 \bigr) &=& \int_0^t
\llVert K_{t-s}\rrVert_{L^2(G)}^2 \E \bigl( \llVert
Z_s\rrVert_{L^2(G)}^2 \bigr)\, \d s\nonumber
\\
&\le&{\e}^{2\beta t} \bigl[\mathcal{N}_\beta(Z)
\bigr]^2\int_0^t {
\e}^{-2\beta(t-s)}\llVert K_{t-s}\rrVert_{L^2(G)}^2 \,\d s
\\
&\le&{\e}^{2\beta t} \bigl[\mathcal{N}_\beta(Z)
\bigr]^2\llVert K\rrVert_{\mathcal
{L}^2_\beta}^2.\nonumber
\end{eqnarray}
The interchange of integrals and expectation is justified by Tonelli's theorem,
thanks to step~1.
Divide by $\exp(-2\beta t)$ and optimize over $t\ge0$ to deduce
(\ref{eqYoung}) whence the lemma.\quad\qed
\end{longlist}\noqed
\end{pf*}

Now we extend the definition of the stochastic convolution as follows:
suppose $K\in\mathcal{L}^2_\beta$ and $Z\in\mathcal{P}^2_\beta$ for\vspace*{1pt}
some $\beta\ge0$. Then we can find simple random fields~$Z^1,Z^2,\ldots$
such that $\lim_{n\to\infty}\mathcal{N}_\beta(Z^n-Z)=0$.
Lemma~\ref{lemYoung} ensures that
%
%
%e6.21 #&#
\begin{equation}
\lim_{n\to\infty}\mathcal{N}_\beta \bigl(K^n
\circledast Z-K\circledast Z \bigr)=0,
\end{equation}
and hence the following result holds.

%
%
%th6.6 #&#
\begin{theorem}\label{thSC}
If $K\in\mathcal{L}^2_\beta(G)$ and $Z\in\mathcal{P}^2_\beta(G)$
for some $\beta\ge0$, then there exists
$K\circledast Z\in\mathcal{P}^2_\beta(G)$ such that
$(K,Z)\mapsto K\circledast Z$ is a.s. a bilinear map
that satisfies (\ref{eqYoung}). This stochastic convolution
$K\circledast Z$ agrees with the Walsh stochastic convolution when
$Z$ is a simple random field.
\end{theorem}

The random field $K\circledast Z$ is
the \emph{stochastic convolution} of $K$ and $Z$. Let us emphasize, however,
that this construction of $K\circledast Z$ produces a stochastic process
$t\mapsto(K\circledast Z)_t$ with values in $L^2(G)$.

%s7 #&#
\section{Proof of Theorem \texorpdfstring{\protect\ref{thexistunique}}{2.1}: Part 1}\label{secproofpart1}

The proof of Theorem~\ref{thexistunique} is divided naturally in two parts:
first, we study the case that $\sigma(0)=0$; after that we visit the
case that
$G$ is compact.
The two cases are handled by different methods.
Throughout this section, we address only the first case, and hence we
assume that
%
%
%e7.1 #&#
\begin{equation}
\sigma(0)=0\qquad\mbox{whence } \bigl\llvert\sigma(z) \bigr\rrvert\le \lip
\llvert z\rrvert\qquad\mbox{for all }z\in\R;
\end{equation}
see (\ref{sigma}).

Our derivation follows ideas of Walsh \cite{Walsh} and Dalang \cite
{Dalang}, but
has novel features as well, since our stochastic convolutions are
not defined as classical (everywhere defined) random fields
but rather as elements of the\vspace*{1pt} space $\bigcup_{\beta\ge0}\mathcal
{P}^2_\beta(G)$.
Therefore, we hash out some of the details of the proof of Theorem~\ref
{thexistunique}.
Throughout, we write $u_t(x)$ in place of $u(t,x)$, as is customary in
the theory of stochastic processes. Thus, let us emphasize
that we \emph{never} write
$u_t$ in place of $\partial u/\partial t$.

Let us follow (essentially) the treatment of Walsh \cite{Walsh}, and
say that
a stochastic process $u:=\{u_t\}_{t\ge0}$ with\vspace*{1pt} values in $L^2(G)$ is a
\emph{mild solution} to (\ref{SHE}) with initial function $u_0\in L^2(G)$,
when $u$ satisfies
%
%
%e7.2 #&#
\begin{equation}
\label{eqmildsoln} u_t = P_tu_0 + \lambda
\bigl( p\circledast\sigma(u) \bigr)_t \qquad\mbox{a.s. for all }t>0,
\end{equation}
viewed as a random dynamical system on $L^2(G)$.\footnote{In
statements such as
this, we sometimes omit writing ``a.s.,'' particularly when the
``almost sure''
assertion is implied clearly.}
Somewhat more precisely, we wish to find a
$\beta\ge0$, sufficiently large, and solve the preceding as a
stochastic integration equation for processes in $\mathcal{P}^2_\beta(G)$,
using that value of $\beta$. Since the spaces $\{\mathcal{P}^2_\beta
(G)\}_{\beta\ge0}$
are nested, there is no unique choice. But as it turns out there is a minimal
acceptable choice for $\beta$, which we also will identify for later purposes.\vspace*{1pt}

The proof proceeds, as usual, by an appeal to Picard iteration. Let\break
$u^{(0)}_t(x):= u_0(x)$ and define iteratively
%
%
%e7.3 #&#
\begin{equation}
\label{equ^n} u^{(n+1)}_t:= P_tu_0+
\lambda \bigl( p\circledast\sigma \bigl(u^{(n)} \bigr)
\bigr)_t,
\end{equation}
for all $n\ge1$. Since
%
%
%e7.4 #&#
\begin{equation}
\mathcal{N}_\beta( P_tu_0 ) \le\sup
_{t\ge0} \llVert P_tu_0\rrVert
_{L^2(G)} = \llVert u_0\rrVert_{L^2(G)}\qquad\mbox{for
all }\beta\ge0,
\end{equation}
and because $\llVert p\rrVert _{\mathcal{L}^2_\beta}^2=\Upsilon(2\beta)$,
it follows from Lemma~\ref{lemYoung} that
%
%
%e7.5 #&#
\begin{eqnarray}
\qquad \mathcal{N}_\beta \bigl( u^{(n+1)} \bigr) &\le&
\llVert u_0\rrVert_{L^2(G)} + \lambda\mathcal{N}_\beta
\bigl( \sigma\circ u^{(n)} \bigr) \biggl( \int_0^\infty{
\e}^{-2\beta s}\llVert p_s\rrVert_{L^2(G)}^2 \,\d s \biggr)^{1/2}
\nonumber\\[-8pt]\\[-8pt]\nonumber
&=& \llVert u_0\rrVert_{L^2(G)} + \lambda
\mathcal{N}_\beta \bigl( \sigma\circ u^{(n)} \bigr)\sqrt{
\Upsilon(2\beta)},
\end{eqnarray}
for all $n\ge1$ and $\beta\ge0$. Next, we apply the Lipschitz
condition of
$\sigma$ together with the fact that $\sigma(0)=0$ in order to deduce
the iterative bound
%
%
%e7.6 #&#
\begin{equation}
\mathcal{N}_\beta \bigl( u^{(n+1)} \bigr) \le\llVert
u_0\rrVert_{L^2(G)} + \mathcal{N}_\beta \bigl(
u^{(n)} \bigr)\lambda\lip\sqrt{\Upsilon(2\beta)}.
\end{equation}
Now we choose $\beta$ somewhat carefully. Let us choose and fix
some $\varepsilon\in(0,1)$, and then define
%
%
%e7.7 #&#
\begin{equation}
\label{eqbeta} \beta:= \frac{1}2\Upsilon^{-1} \biggl(
\frac{1}{(1+\varepsilon)^2
\lambda^2\lip^2} \biggr),
\end{equation}
which leads to the identity
$\lambda\lip\sqrt{\Upsilon(2\beta)}= (1+\varepsilon)^{-1}$, whence
%
%
%e7.8 #&#
\begin{equation}
\mathcal{N}_\beta \bigl( u^{(n+1)} \bigr) \le\llVert
u_0\rrVert_{L^2(G)} + \frac{1}{(1+\varepsilon)}\mathcal{N}_\beta
\bigl( u^{(n)} \bigr).
\end{equation}
Since $\mathcal{N}_\beta(u_0) = \llVert u_0\rrVert _{L^2(G)}$, it
follows that
%
%
%e7.9 #&#
\begin{equation}
\sup_{n\ge0}\mathcal{N}_\beta \bigl( u^{(n)}
\bigr) \le\frac{1+\varepsilon}{\varepsilon} \llVert u_0\rrVert_{L^2(G)}.
\end{equation}
The same value of $\beta$ can be applied in a similar way
in order to deduce that
%
%
%e7.10 #&#
\begin{equation}
\mathcal{N}_\beta \bigl( u^{(n+1)}-u^{(n)} \bigr) \le
\frac{1}{1+\varepsilon}\mathcal{N}_\beta \bigl( u^{(n)}-u^{(n-1)}
\bigr).
\end{equation}
This shows, in particular, that $\sum_{n=0}^\infty\mathcal{N}_\beta
(u^{(n+1)}
-u^{(n)})<\infty$, whence there exists $u$ such that
$\lim_{n\to\infty}\mathcal{N}_\beta(u^{(n)}-u)=0$. Since
%
%
%e7.11 #&#
\begin{eqnarray}
&& \mathcal{N}_\beta \bigl( p\circledast \bigl[\sigma
\bigl(u^{(n)} \bigr)- \sigma(u ) \bigr] \bigr) \nonumber
\\
&&\qquad \le \lambda
\mathcal{N}_\beta \bigl(\sigma \bigl(u^{(n)} \bigr)- \sigma(u )
\bigr)\cdot \biggl(\int_0^\infty{
\e}^{-2\beta s}\llVert p_s\rrVert_{L^2(G)}^2 \,\d s \biggr)^{1/2}
\\
&&\qquad \le\lambda\lip\cdot\,\mathcal{N}_\beta \bigl( u^{(n)}-u
\bigr) \sqrt{\Upsilon(2\beta)},\nonumber
\end{eqnarray}
it follows that the stochastic convolution $p\circledast\sigma(u^{(n)})$
converges in norm $\mathcal{N}_\beta$ to the stochastic convolution
$p\circledast\sigma(u)$. Thus, it follows that $u$ solves the stochastic
heat equation and the $L^2$ moment bound on $u$ is a consequence of the
fact that
$\mathcal{N}_\beta(u)\le(1+\varepsilon)\varepsilon^{-1}\llVert
u_0\rrVert _{L^2(G)}$,
for the present choice
of $\beta$. The preceding can be unscrambled as follows:
%
%
%e7.12 #&#
\begin{equation}
\label{eqNu} \qquad\E \bigl(\llVert u_t\rrVert_{L^2(G)}^2
\bigr) \le\frac{(1+\varepsilon)^2}{\varepsilon^2} \llVert u_0\rrVert _{L^2(G)}^2
\exp \biggl\{\frac{t}2\Upsilon^{-1} \biggl( \frac{1}{(1+\varepsilon)^2\lambda^2\lip^2}
\biggr) \biggr\},
\end{equation}
for all $\varepsilon\in(0,1)$ and $t\ge0$. Of course,
(\ref{eqL2bdd}) is a ready consequence. This proves
the existence of the right sort of mild solution to (\ref{SHE}).

The proof of uniqueness follows the ideas of Dalang \cite{Dalang} but
computes norms
in $L^2(G)$ rather than pointwise norms. To be more specific, suppose
$v$ is another solution that satisfies (\ref{eqL2bdd}) for some finite
constant $c\ge0$. Then of course $v$ satisfies~(\ref{eqL2bdd})
also when $c$ is replaced by any other larger constant. Therefore,
there exists $\beta\ge c\ge0$ such that $u,v\in\mathcal{P}^2_\beta$
(for the same $\beta$). A calculation, very much similar to those
we made earlier for Picard's iteration, shows that
%
%
%e7.13 #&#
\begin{equation}
\mathcal{N}_\beta(u-v) \le\lambda\lip\cdot\,\mathcal{N}_\beta
(u-v)\cdot\sqrt{\Upsilon(2\beta)},
\end{equation}
whence it follows that the $L^2(G)$-valued stochastic processes
$\{u_t\}_{t\ge0}$ and $\{v_t\}_{t\ge0}$ are modifications of one another.
This completes the proof.\qed

%s8 #&#
\section{Proof of Theorem \texorpdfstring{\protect\ref{thexistunique}}{2.1}: Part~2}\label{secproofpart2}
It remains to prove theorem in the case
that $G$ is compact. If, additionally, $\sigma(0)=0$, then the
existence and
uniqueness of a solution follows from the proof of the
noncompact case. That proof states, in an {a priori} sense, that if
$u_0\in L^2(G)$ and
$\sigma(0)=0$, then $u_t\in L^2(G)$ for all $t>0$ as well. This
property is not in general
true. Therefore, we need to proceed otherwise. Our approach is to
reduce the problem
to the case that $u_0\in C_c(G)$, by approximation. Then we show that, in
the case that $u_0\in C_c(G)$,
(\ref{SHE}) has a pointwise (random field) solution that has the property that
%
%
%e8.1 #&#
\begin{equation}
C_T:=\sup_{t\in[0,T]}\sup_{x\in G}\E
\bigl( \bigl\llvert u_t(x) \bigr\rrvert^2 \bigr)<\infty
\qquad\mbox{for all }T>0.
\end{equation}
It then follows from Tonelli's theorem that $\sup_{t\in[0,T]}\mathscr
{E}_t(\lambda)
\le C_T<\infty$, since $m_G(G)=1$ in the compact case.

The actual proof requires a number of small technical steps.

Recall the norms $\mathcal{N}_\beta$. We now introduce a slightly
different family of norms that were introduced earlier in Foondun and
Khoshnevisan \cite{FoondunKh}.

%
%
%de8.1 #&#
\begin{definition}
For every $\beta\ge0$ and for all everywhere-defined
random fields
$Z:=\{Z_t(x)\}_{t\ge0,x\in G}$, we define
%
%
%e8.2 #&#
\begin{equation}
\mathcal{M}_\beta(Z):= \sup_{t\ge0}\sup
_{x\in G} \bigl\{ {\e}^{-2\beta t}\E \bigl( \bigl\llvert
Z_t(x) \bigr\rrvert^2 \bigr) \bigr\}^{1/2}.
\end{equation}
\end{definition}

We\vspace*{-1pt} can define predictable random fields $\mathcal{P}^\infty_\beta(G)$
with respect to the preceding norms,
just as we defined spaces $\mathcal{P}^2_\beta(G)$ of predictable random
fields for $\mathcal{N}_\beta$ in Definition~\ref{defP}.

%
%
%de8.2 #&#
\begin{definition}
For every $\beta\ge0$, we define $\mathcal{P}^\infty_\beta(G)$ to
be the
completion of the collection of simple random fields in the norm
$\mathcal{M}_\beta$.
We may observe that: (i)~Every\vspace*{1pt} $\mathcal{P}^\infty_\beta(G)$ is a
Banach space,
once endowed with norm $\mathcal{M}_\beta$; and (ii)~if $\alpha
<\beta$, then
$\mathcal{P}^\infty_\alpha(G)\subseteq\mathcal{P}^\infty_\beta(G)$.
\end{definition}

Note that $\mathcal{M}_\beta$ is a larger norm than $\mathcal
{N}_\beta$
on $\mathcal{P}^\infty_\beta(G)$,
since $G$ is compact. Indeed, because $m_G(G)=1$ it follows that
$\mathcal{N}_\beta(Z)\le\mathcal{M}_\beta(Z)$ for all $Z\in
\mathcal{P}^\infty_\beta(G)$.\vspace*{1pt}

The stochastic convolution $K\circledast Z$ can be
defined for $Z\in\mathcal{P}^\infty_\beta(G)$
as well, just as one does it for $Z\in\mathcal{P}^2_\beta(G)$
(Theorem~\ref{thSC}). The
end result is the following.

%
%
%th8.3 #&#
\begin{theorem}\label{thSC2}
If\vspace*{1pt} $K\in\mathcal{L}^2_\beta(G)$ and $Z\in\mathcal{P}^\infty_\beta(G)$
for some $\beta\ge0$, then there exists
$K\circledast Z\in\mathcal{P}^\infty_\beta(G)$ such that
$(K,Z)\mapsto K\circledast Z$ is a.s. a bilinear map
that satisfies the stochastic Young inequality,
%
%
%e8.3 #&#
\begin{equation}
\mathcal{M}_\beta(K\circledast Z)\le\mathcal{M}_\beta(Z)
\cdot\llVert K\rrVert_{\mathcal{L}^2_\beta(G)}.
\end{equation}
This stochastic convolution
$K\circledast Z$ agrees with the Walsh stochastic convolution when
$Z$ is a simple random field.
\end{theorem}

The proof of Theorem~\ref{thSC2}
follows the same general pattern of the proof of Theorem~\ref{thSC}
but one has to make a few adjustments that, we feel, are routine. Therefore,
we omit the details. However, we would like to emphasize that this
stochastic convolution
is not always the same as the one that was constructed in the previous sections.
In particular, let us note that if $K\in\mathcal{L}^2_\beta(G)$
and $Z\in\mathcal{P}^\infty_\beta(G)$ for some
$\beta\ge0$, then $(K\circledast Z)_t(x)$ is a well-defined uniquely defined
random variable for all $t>0$ and $x\in G$. This should be compared to
the fact that $(K\circledast Z)_t$ is defined only as an element of
$L^2(G)$ when $Z\in\mathcal{P}^2_\beta(G)$.

The next result shows that
(\ref{SHE}) has a a.s.-unique mild pointwise solution $u$ whenever $u_0\in
L^\infty(G)$,
in the sense that
$u$ is the a.s.-unique solution to the equation
%
%
%e8.4 #&#
\begin{equation}
u_t(x) = (P_tu_0) (x) + \bigl(p
\circledast\sigma(u) \bigr)_t(x),
\end{equation}
valid a.s. for every $x\in G$ and $t>0$. The preceding
stochastic convolution is understood to be the one that we just
constructed in this section.
Among other things, the following tacitly ensures that the said
stochastic convolution
is well defined.

%
%
%th8.4 #&#
\begin{theorem}\label{thexistuniquebdd}
Let $G$ be an LCA group, and $\{X_t\}_{t\ge0}$ be a L\'evy process on~$G$.
If $u_0\in L^\infty(G)$, then
for every $\lambda>0$, the stochastic heat equation (\ref{SHE})
has a mild pointwise solution $u$ that satisfies the following:
there exists a finite constant $b\ge1$ that yields the energy inequality
%
%
%e8.5 #&#
\begin{equation}
\label{eqLinftybdd} \sup_{x\in G}\E \bigl( \bigl\llvert
u_t(x) \bigr\rrvert^2 \bigr) \le b{\e}^{bt}
\qquad\mbox{for every }t\ge0.
\end{equation}
Moreover, if $v$ is any mild solution that satisfies (\ref{eqL2bdd})
as well as $v_0=u_0$,
then $\P\{u_t(x)=v_t(x)\}=1$ for all $t\ge0$ and $x\in G$.
\end{theorem}

One can model a proof of Theorem~\ref{thexistuniquebdd}
after the already-proved portion of Theorem~\ref{thexistunique}
[i.e., in the case that $\sigma(0)=0$], but use the norm $\mathcal
{M}_\beta$
in place of $\mathcal{N}_\beta$.
In fact, such a proof will imply that
(\ref{eqLinftybdd}) has a solution that is in $L^\infty(G)$ at all times
as long as $u_0\in L^\infty(G)$, even if $G$ is not compact and
$\sigma(0)$ is not $0$.
When $G=\R$, the latter facts are also contained within
the theory of Dalang \cite{Dalang}.
For these reasons, we omit the proof of Theorem~\ref
{thexistuniquebdd}. But let us emphasize that
since $u$ is a random field in the sense of the present paper, (\ref
{eqLinftybdd})
and Fubini's theorem together imply that if $u_0\in L^\infty(G)$, then
%
%
%e8.6 #&#
\begin{equation}
\E \bigl(\llVert u_t\rrVert_{L^2(G)}^2 \bigr) \le
b{\e}^{bt}m_G(G).
\end{equation}
Now let us recall that for our present purposes $G$ is compact,
and hence \mbox{$m_G(G)=1$}. It follows from these conditions that
the solution $u_t$ is also in $L^2(G)$,
for all $t>0$, as long as $u_0\in L^\infty(G)$.\footnote{This
property can fail
when $G$ is not compact and $\sigma(0)$ is not zero.
For example, if $u_0=0$, $G=\R$, and $\sigma\equiv1$ (the linear
stochastic heat equation),
then there is a unique solution that is in $L^\infty(\R)$ at all
times but there
is no solution that is in $L^2(\R)$ at any time $t>0$.}

Now we begin our proof of Theorem~\ref{thexistunique} in the
case that $G$ is compact, an assumption which we assume for the
remainder of the section.

Our normalization of Haar measure ensures that $m_G(G)=1$ in the
present compact case.
Consequently, $L^\infty(G)\subset L^2(G)$, and hence
if $u_0\in L^\infty(G)$, then (\ref{SHE}) has a random field solution,
with values in $L^2(G)\cap L^\infty(G)$ at all times, such that
%
%
%e8.7 #&#
\begin{equation}
\label{eqLinftyL2} \E \bigl(\llVert u_t\rrVert_{L^2(G)}^2
\bigr) \le b{\e}^{bt}.
\end{equation}
We also find, {a priori}, that $u\in\mathcal{P}^2_\beta(G)$ for
all sufficiently large $\beta$.
This proves the theorem when $G$ is compact and $u_0\in L^\infty(G)$.

In fact, we can now use the {a priori}
existence bounds that we just developed in order to argue, somewhat as
in the Walsh theory,
and see that [in this case where $u_0\in L^\infty(G)$]
%
%
%e8.8 #&#
\begin{eqnarray}\label{eqisometry}
\E \bigl( \bigl\llvert u_t(x) \bigr\rrvert
^2 \bigr) &=& \bigl\llvert(P_tu_0) (x) \bigr
\rrvert^2
\nonumber\\[-8pt]\\[-8pt]
&&{} + \lambda^2 \int_0^t\d s\int_G m_G(\d y) \bigl[p_{t-s}
\bigl(yx^{-1} \bigr) \bigr]^2\E \bigl( \bigl\llvert\sigma
\bigl(u_s(y) \bigr) \bigr\rrvert^2 \bigr),\nonumber
\end{eqnarray}
for all $t>0$ and $x\in G$.
But we will not need this formula at this time. Instead, let us observe the
following variation: if
$v$ solves (\ref{SHE})---for the same white noise $\xi$---with $v_0\in
L^\infty(G)$,
then
%
%
%e8.9 #&#
\begin{eqnarray}
&& \E \bigl( \bigl\llvert u_t(x)-v_t(x) \bigr\rrvert
^2 \bigr)\nonumber
\\
&&\qquad  = \bigl\llvert(P_tu_0) (x) -
(P_tv_0) (x) \bigr\rrvert^2\nonumber
\\
&&\quad\qquad{} + \lambda^2\int_0^t \d s
\int_Gm_G(\d y) \bigl[p_{t-s}
\bigl(yx^{-1} \bigr) \bigr]^2\E \bigl( \bigl\llvert\sigma
\bigl(u_s(y) \bigr)-\sigma \bigl(v_s(y) \bigr) \bigr
\rrvert ^2 \bigr)
\\
&&\qquad \le \bigl\llvert(P_tu_0) (x) - (P_tv_0)
(x) \bigr\rrvert^2\nonumber
\\
&&\quad\qquad{} +\lambda^2\lip^2\cdot\int_0^t
\d s\int_Gm_G( \d y) \bigl[p_{t-s}
\bigl(yx^{-1} \bigr) \bigr]^2\E \bigl( \bigl\llvert
u_s(y)-v_s(y) \bigr\rrvert ^2 \bigr).\nonumber
\end{eqnarray}
Since each $P_t$ is a linear contraction on $L^2(G)$, we may integrate
both sides of the preceding inequality in order
to deduce the following from Fubini's
theorem: for every $\beta\ge0$,
%
%
%e8.10 #&#
\begin{eqnarray}
&& \E \bigl( \llVert u_t-v_t\rrVert_{L^2(G)}^2
\bigr)\nonumber
\\
&&\qquad  \le\llVert u_0-v_0\rrVert_{L^2(G)}^2
+ \lambda^2\lip^2\cdot\int_0^t
\llVert p_{t-s}\rrVert_{L^2(G)}^2\E \bigl(\llVert
u_s-v_s\rrVert_{L^2(G)}^2 \bigr)
\\
&&\qquad \le\llVert u_0-v_0\rrVert_{L^2(G)}^2
+ \lambda^2\lip^2{\e}^{2\beta t} \bigl[
\mathcal{N}_\beta(u-v) \bigr]^2\cdot\Upsilon(2\beta).\nonumber
\end{eqnarray}
In particular,
%
%
%e8.11 #&#
\begin{equation}
\bigl[\mathcal{N}_\beta(u-v) \bigr]^2 \le\llVert
u_0-v_0\rrVert_{L^2(G)}^2 +
\lambda^2\lip^2 \bigl[\mathcal{N}_\beta(u-v)
\bigr]^2\Upsilon(2\beta).
\end{equation}
Owing to (\ref{eqLinftyL2}), we know that $\mathcal{N}_\beta
(u-v)<\infty$
if $\beta$ is sufficiently large. By the dominated convergence theorem,
$\lim_{\beta\uparrow\infty}\Upsilon(2\beta)=0$, whence we have
%
%
%e8.12 #&#
\begin{equation}
\lambda^2\lip^2\Upsilon(2\beta)\le\nicefrac12 \qquad
\mbox{for all }\beta\mbox{ large enough.}
\end{equation}
This shows that
%
%
%e8.13 #&#
\begin{equation}
\label{equ-v} \mathcal{N}_\beta(u-v) \le\operatorname{const}\cdot\,\llVert u_0-v_0\rrVert_{L^2(G)},
\end{equation}
for all $u_0,v_0\in L^\infty(G)$
and an implied constant that is finite and depends only on $(\lambda,\lip,\Upsilon)$.

Now that we have proved (\ref{equ-v}), we can complete the proof of
Theorem~\ref{thexistunique} (in the case that $G$ is compact) as follows:
suppose $u_0\in L^2(G)$. Since $C_c(G)$ is dense in $L^2(G)$, we can find
$u_0^{(1)},u_0^{(2)},\ldots\in C_c(G)$ such that $u_0^{(n)}\to u_0$
in $L^2(G)$ as $n\to\infty$. Let
$u^{(n)}:=\{u^{(n)}_t(x)\}_{t\ge0,x\in G}$ denote the solution to (\ref{SHE})
starting at $u^{(n)}_0$.
Equation (\ref{equ-v}) shows that $\{u^{(n)}\}_{n=1}^\infty$ is a
Cauchy sequence
in $\mathcal{P}^2_\beta(G)$ provided that $\beta$ is
chosen to be sufficiently large (but fixed). Therefore, $w:=\lim_{n\to
\infty} u^{(n)}$ exists
in $\mathcal{P}^2_\beta(G)$. Lemma~\ref{lemYoung}\vspace*{2pt} ensures that
$p\circledast u^{(n)}$ converges to $p\circledast w$, and hence
$w$ solves (\ref{SHE}) starting at $u_0$. This proves existence. Uniqueness
is proved
by similar approximation arguments.

%s9 #&#
\section{Proof of Proposition \texorpdfstring{\protect\ref{prlinear}}{2.3}}\label{secproofprop}
First, consider the case that $u_0\in L^\infty(G)$. In that case,
we may apply (\ref{eqisometry}) in order to see that the solution
$u$ is defined pointwise and satisfies
%
%
%e9.1 #&#
\begin{equation}
\E \bigl( \bigl\llvert u_t(x) \bigr\rrvert^2 \bigr) \le
\bigl\llvert(P_tu_0) (x) \bigr\rrvert^2 +
\lambda^2\llVert\sigma\rrVert_{L^\infty(\R)}^2 \int
_0^t\llVert p_s\rrVert
_{L^2(G)}^2 \,\d s.
\end{equation}
Since $\int_0^t\llVert p_s\rrVert _{L^2(G)}^2 \,\d s=\int_0^t\bar
{p}_s({\e}_G)\, \d s
\le{\e}\Upsilon(1/t)<\infty$ [(\ref{eqpUpsilon})] and $G$ is compact,
the $L^2(G)$-contractive
property of $P_t$ yields
%
%
%e9.2 #&#
\begin{equation}
\bigl[\mathscr{E}_t(\lambda) \bigr]^2 = \E \bigl(
\llVert u_t\rrVert_{L^2(G)}^2 \bigr) \le\llVert
u_0\rrVert_{L^2(G)}^2 + {\e} \lambda^2
\llVert\sigma\rrVert^2_{L^\infty(\R)} \Upsilon(1/t).
\end{equation}

If $u$ is known to be only in $L^2(G)$, then by density we can
find for every $\varepsilon>0$ a function $v\in L^\infty(G)$
such that $\llVert u_0-v_0\rrVert _{L^2(G)}\le\varepsilon$. The
preceding paragraph and~(\ref{equ-v}) together yield
%
%
%e9.3 #&#
\begin{eqnarray}
\bigl[\mathscr{E}_t(\lambda) \bigr]^2 &\le&2{
\e}^{2\beta t} \bigl[\mathcal{N}_\beta(u-v) \bigr]^2 + 2
\bigl(\llVert v_0\rrVert_{L^2(G)}^2 + {\e}
\lambda^2\llVert\sigma\rrVert^2_{L^\infty(\R)}
\Upsilon(1/t) \bigr)\hspace*{-30pt}
\nonumber\\[-8pt]\\[-8pt]\nonumber
&\le&\operatorname{const}\cdot\,2{\e}^{2\beta t}\varepsilon^2 + 2
\bigl(2\llVert u_0\rrVert_{L^2(G)}^2 + 2
\varepsilon^2+ {\e}\lambda^2\llVert\sigma\rrVert
^2_{L^\infty(\R)}\Upsilon(1/t) \bigr).\hspace*{-30pt}
\end{eqnarray}
This is more than enough to show that
$\mathscr{E}_t(\lambda)=O(\lambda)$ for all $t>0$. In fact, it
yields also the quantitative bound,
%
%
%e9.4 #&#
\begin{equation}
\mathscr{E}_t(\lambda)\le\operatorname{const} \cdot\, \bigl( \llVert
u_0\rrVert_{L^2(G)} + \lambda\llVert\sigma\rrVert
_{L^\infty(\R)} \sqrt{\Upsilon(1/t)} \bigr),
\end{equation}
for a finite universal constant.
This completes the first portion of the proof.

If $\llvert\sigma\rrvert$ is bounded uniformly from below, then we reduce
the problem to the case that $u_0\in L^\infty(G)$ just as we did in
the first half, using (\ref{equ-v}), and then apply (\ref{eqisometry})
in order to see that [in the case that $u_0\in L^\infty(G)$],
%
%
%e9.5 #&#
\begin{equation}
\E \bigl( \bigl\llvert u_t(x) \bigr\rrvert^2 \bigr) \ge
\inf_{z\in G} \bigl\llvert u_0(z) \bigr\rrvert
^2 + \lambda^2 \inf_{z\in\R} \bigl\llvert
\sigma(z) \bigr\rrvert^2\cdot\int_0^t
\llVert p_s\rrVert_{L^2(G)}^2 \,\d s.
\end{equation}
We will skip the remaining details on how one makes the transition from
considerations of initial values $u_0\in L^\infty(G)$
to initial values $u_0\in L^2(G)$: this issue has been dealt with
already in the first
half of the proof. Instead, let us conclude the proof by observing that
the preceding
is consistent, since $\int_0^t\llVert p_s\rrVert _{L^2(G)}^2 \,\d s>0$,
for if this
integral were zero for all $t$ then the proof would fail.
But because $G$ is compact and $m_G$ is a probability measure on $G$,
Jensen's inequality reveals that $\llVert p_s\rrVert _{L^2(G)}^2\ge
\llVert p_s\rrVert _{L^1(G)}^2=1$.
Therefore,\vspace*{1pt} $\int_0^t\llVert p_s\rrVert _{L^2(G)}^2 \,\d s\ge t$ is
positive when
$t$ is
positive, as was advertised.

%s10 #&#
\section{Condition \texorpdfstring{(\protect\ref{D})}{(D)} and local times}\label{seclocal}
Dalang's condition (\ref{D}) is connected intimately to the theory of local times
for L\'evy processes. This connection was pointed out in Foondun, Khoshnevisan
and Nualart \cite{FKN} when $G=\R$; see also
Eisenbaum et al. \cite{EisenbaumFoondunKh}. Here, we describe how
one can extend that connection to the present, more general, setting
where $G$
is an LCA group.

Let $Y:=\{Y_t\}_{t\ge0}$ be an independent copy of $X$, and consider the
stochastic process
%
%
%e10.1 #&#
\begin{equation}
S_t:= X_t Y_t^{-1}\qquad(t
\ge0).
\end{equation}
It is easy to see that $S:=\{S_t\}_{t\ge0}$ is a L\'evy process with
characteristic function
%
%
%e10.2 #&#
\begin{equation}
\E(S_t,\chi) = {\e}^{-2t\Re\Psi(\chi)}\qquad\mbox{for all }t\ge0 \mbox{
and }\chi\in G^*,
\end{equation}
where $\Psi$ denote the L\'evy--Khintchine exponent, or characteristic
exponent, of
the L\'evy process $\{X_t\}_{t\ge0}$.
The process $S$ is called the L\'evy symmetrization of $X$; the
nomenclature is motivated by the fact that each $S_t$ is a
\emph{symmetric} random variable in the sense that
$S_t$ and $S_t^{-1}$ have the same distribution for all $t\ge0$.

Let $J$ denote the \emph{weighted occupation measure} of $S$, that is,
%
%
%e10.3 #&#
\begin{equation}
J(A):= \int_0^\infty\1_A(S_s){
\e}^{-s} \,\d s,
\end{equation}
for all Borel sets $A\subset G$. It is easy to see that
%
%
%e10.4 #&#
\begin{equation}
\hat{J}(\chi):= \int_G (x,\chi) J(\d x) = \int
_0^\infty(S_s,\chi){\e}^{-s}\,
\d s \qquad \bigl(\chi\in G^* \bigr),
\end{equation}
whence
%
%
%e10.5 #&#
\begin{equation}
\E \bigl( \bigl\llvert\hat{J}(\chi) \bigr\rrvert^2 \bigr) = 2\int
_0^\infty{\e}^{-t} \,\d t\int
_0^t{\e}^{-s} \,\d s\, \E \bigl[
(S_s,\chi) \overline{(S_t,\chi)} \bigr].
\end{equation}
For every $s,t\ge0$ and for all characters $\chi\in G^*$,
%
%
%e10.6 #&#
\begin{equation}
(S_s,\chi) \overline{(S_t,\chi)} =
\chi(S_s) \chi \bigl(S_t^{-1} \bigr) =\chi
\bigl(S_sS_t^{-1} \bigr)= \bigl(
S_sS_t^{-1} \bigr) (\chi).
\end{equation}
Note that $S_sS_t^{-1}=(S_tS_s^{-1})^{-1}$, and that
the distribution of $S_tS_s^{-1}$ is the same as the
distribution of $S_{t-s}$ for $t\ge s\ge0$.
Since $S_{t-s}$ has the same distribution as that of~$S_{t-s}^{-1}$, by
the symmetry of $S$, it follows that
%
%
%e10.7 #&#
\begin{eqnarray}
\E \bigl( \bigl\llvert\hat{J}(\chi) \bigr\rrvert^2 \bigr) &=& 2\int
_0^\infty{\e}^{-t} \,\d t\int
_0^t{\e}^{-s} \,\d s\, \E \bigl[
(S_{t-s},\chi) \bigr]\nonumber
\\
&=& 2\int_0^\infty{\e}^{-s} \,\d s\int
_s^\infty{\e}^{-t} \d t\, {
\e}^{-(t-s)\Re\Psi(\chi)}
\\
&=&\frac{1}{1+2\Re\Psi(\chi)},\nonumber
\end{eqnarray}
for every $\chi\in G^*$. Therefore,
%
%
%e10.8 #&#
\begin{equation}
\E \bigl( \llVert\hat{J}\rrVert_{L^2(G^*)}^2 \bigr) = \int
_{G^*} \biggl(\frac{1}{1+2\Re\Psi(\chi)} \biggr) m_{G^*}(\d
\chi)= \Upsilon(1).
\end{equation}

In particular, we have proved that Dalang's condition (\ref{D}) is equivalent
to the condition that
%
%
%e10.9 #&#
\begin{equation}
\ell(x):= \frac{\d J}{\d m_G}(x)\qquad\mbox{exists and is in }L^2(\P
\times m_G),
\end{equation}
and in this case,
%
%
%e10.10 #&#
\begin{equation}
\label{eqellPlancherel} \E \bigl( \llVert\ell\rrVert_{L^2(G)}^2 \bigr)
= \E \bigl( \llVert\hat{J}\rrVert_{L^2(G^*)}^2 \bigr) =
\Upsilon(1),
\end{equation}
thanks to Plancherel's theorem.
For real-valued L\'evy processes, this observation is due
essentially to Hawkes \cite{Hawkes}.

The random field $\ell$ is called the \emph{local times} of $\{S_t\}
_{t\ge0}$;
$\ell$ has, by its very definition, the property that it is a random
probability
function on $G$ such that
%
%
%e10.11 #&#
\begin{equation}
\int_G f\ell\,\d m_G = \int
_0^\infty f(S_t){\e}^{-t} \,\d
t\qquad\mbox{a.s.},
\end{equation}
for all nonrandom functions $f\in L^2(G)$.

Let us now return to the following remark that was made in the \hyperref[sec1]{Introduction}.

%
%
%le10.1 #&#
\begin{lemma}\label{lemexistG-discrete}
Dalang's condition (\ref{D}) holds whenever $G$ is discrete.
\end{lemma}

This lemma was shown to hold as a consequence of Pontryagin--van Kampen
duality. We can now understand this lemma probabilistically.

\begin{pf*}{A probabilistic proof of Lemma~\ref{lemexistG-discrete}}
When $G$ is discrete, local times always exist and are described via
%
%
%e10.12 #&#
\begin{equation}
\ell(x):= \int_0^\infty\1_{\{x\}}(S_t){
\e}^{-t} \,\d t \qquad(x\in G).
\end{equation}
In light of (\ref{eqellPlancherel}),
it remains to check only that $\ell\in L^2(\P\times m_G)$,
since it is evident that $\ell=\d J/\d m_G$ in this case.
But since $m_G$ is the counting measure on $G$,
%
%
%e10.13 #&#
\begin{eqnarray}
\Upsilon(1) &=& \llVert\ell\rrVert_{L^2(\P\times m_G)}^2\nonumber
\\
&=& 2\sum_{x\in G} \int_0^\infty{
\e}^{-s} \,\d s\int_s^\infty{
\e}^{-t} \,\d t\, \P\{S_s=x,S_t=x\}
\\
&=&2\int_0^\infty{\e}^{-s} \,\d s\int
_s^\infty{\e}^{-t} \,\d t\, \P
\{S_{t-s}={\e}_G\},\nonumber
\end{eqnarray}
where ${\e}_G$ denotes the identity element in $G$.
Since $\P\{S_{t-s}={\e}_G\}\le1$, it follows readily that
$\Upsilon(1)<\infty$, whence follows condition (\ref{D}).
\end{pf*}

%s11 #&#
\section{Group invariance of the excitation indices}\label{secgroup}
The principal aim of this section is to prove that
the noise excitation indices $\overline{\mathfrak{e}}(t)$
and $\underline{\mathfrak{e}}(t)$ are ``group invariants.''
In order to do this, we need to apply some care, but it is easy to
describe informally what group invariance means: if we apply
a topological isomorphism to $G$, then we do not change the values of
$\overline{\mathfrak{e}}(t)$ and $\underline{\mathfrak{e}}(t)$.

%
%
%de11.1 #&#
\begin{definition}
Recall that two LCA groups $G$ and $\Gamma$ are \emph{isomorphic}
(as topological groups) if there exists a homeomorphic homomorphism
$h\dvtx G\to\Gamma$. We will denote by $\operatorname{Iso}(G,\Gamma)$
the collection
of all such topological isomorphisms, and write ``$G\cong\Gamma$'' when
$\operatorname{Iso}(G,\Gamma)\neq\varnothing$; that is precisely
when $G$ and
$\Gamma$ are isomorphic to one another.
\end{definition}

Throughout this section, we consider two LCA groups $G\cong\Gamma$.

It is easy to see that if $h\in\operatorname{Iso}(G,\Gamma)$, then
$m_\Gamma\circ h$
is a translation-invariant Borel measure on $G$ whose total mass agrees
with the total
mass of $m_G$. Therefore, we can find a constant
$\mu(h)\in(0,\infty)$ such that
%
%
%e11.1 #&#
\begin{equation}\label{eq11.1}
m_\Gamma\circ h =\mu(h) m_G \qquad\mbox{for all }h\in
\operatorname{Iso}(G,\Gamma).
\end{equation}

%
%
%de11.2 #&#
\begin{definition}\label{defmodulus}
We refer to $\mu\dvtx \operatorname{Iso}(G,\Gamma)\to(0,\infty)$
as the \emph{modulus function}, and $\mu(h)$ as the
\emph{modulus of an isomorphism} $h\in\operatorname{Iso}(G,\Gamma)$.
In particular, we say that $G$ is \emph{unimodular} when $\mu(h)=1$.
\end{definition}

This definition is motivated by the following: since $G\cong G$,
the collection $\operatorname{Aut}(G):=\operatorname{Iso}(G,G)$ of
all automorphisms
of $G$ is never empty. Recall that $\operatorname{Aut}(G)$ is in general
a non-Abelian group endowed with group product $h\circ g$ (composition)
and group inversion $h^{-1}$ (functional inversion). It is then easy to
see that
$\mu$ is a homomorphism from $\operatorname{Aut}(G)$ into the
multiplicative positive reals
$\R_{>0}^\times$; that is, that $\mu(h\circ g)=\mu(h)\mu(g)$ and
$\mu(h^{-1})=1/\mu(h)$ for every $h,g\in\operatorname{Aut}(G)$.
Thus, the Definition
\ref{defmodulus} of a unimodular group agrees with the usual one when
$\Gamma=G$.

The following simple lemma is an immediate consequence of our standard
normalization of Haar measures and states that compact and/or discrete
LCA groups are unimodular. But it is worth recording.

%
%
%le11.3 #&#
\begin{lemma}
Every element of $\operatorname{Iso}(G,\Gamma)$
is measure preserving when $G$ is either compact or discrete. In
other words, if $G$ is compact or discrete, then so is~$\Gamma$, and
$\mu(h)=1$ for every $h\in\operatorname{Iso}(G,\Gamma)$.
\end{lemma}

Next, let $\xi$
denote a space--time white noise on $\R_+\times G$. Given a function
$h\in\operatorname{Iso}(G,\Gamma)$, we may define a random set function
$\xi_h$ on $\Gamma$ as follows:
%
%
%e11.2 #&#
\begin{equation}
\xi_h(A\times B):= \sqrt{\mu(h)} \xi \bigl(A\times
h^{-1}(B) \bigr),
\end{equation}
for all Borel sets $A\subset\R_+$ and $B\subset\Gamma$ with finite
respective measures $\operatorname{Leb}(A)$ and $m_G(B)$. In this way,
we find that
$\xi_h$ is a totally scattered Gaussian random measure on $\R_+\times
\Gamma$
with control measure $\operatorname{Leb}\times m_\Gamma$. Moreover,
%
%
%e11.3 #&#
\begin{eqnarray}
\E \bigl( \bigl\llvert\xi_h(A\times B) \bigr\rrvert^2
\bigr) &=& \mu(h) \operatorname{Leb}(A) \bigl(m_G\circ
h^{-1} \bigr) (B)
\nonumber\\[-8pt]\\[-8pt]\nonumber
&=& \operatorname{Leb}(A) m_\Gamma(B).
\end{eqnarray}
In other words, we have verified the following simple fact.

%
%
%le11.4 #&#
\begin{lemma}\label{lemmuWN}
Let $\xi$ denote a space--time white noise on
$\R_+\times G$. Then $\xi_h$ is a white noise
on $\R_+\times\Gamma$ for every $h\in\operatorname{Iso}(G,\Gamma)$.
\end{lemma}

Note, in particular, that we can solve SPDEs on $(0,\infty)\times
\Gamma$
using the space--time white noise $\xi_h$. We will return to this matter
shortly.

If $f\in L^2(G)$ and $h\in\operatorname{Iso}(G,\Gamma)$,
then $f\circ h^{-1}$ can be defined uniquely as an element of
$L^2(\Gamma)$ as well as pointwise. Here is how: first, let us
consider $f\in C_c(G)$,
in which case $f\circ h^{-1}\dvtx \Gamma\to\R$ is defined pointwise and is
in $C_c(\Gamma)$. Next, we observe that
%
%
%e11.4 #&#
\begin{eqnarray}
\label{eqL2compose} \bigl\llVert f\circ h^{-1} \bigr\rrVert
_{L^2(\Gamma)}^2 &=& \int_\Gamma \bigl\llvert f
\bigl(h^{-1}(x) \bigr) \bigr\rrvert^2 m_\Gamma(\d
x)\nonumber
\\
&=& \int_G \bigl\llvert f(y) \bigr\rrvert^2
(m_\Gamma\circ h) (\d y)
\\
&=&\mu(h)\llVert f\rrVert_{L^2(G)}^2.\nonumber
\end{eqnarray}
Since $C_c(G)$ is dense in $L^2(G)$, the preceding constructs
uniquely $f\circ h^{-1}\in L^2(\Gamma)$ for every topological
isomorphism $h\dvtx G\to\Gamma$.
Moreover, it follows that (\ref{eqL2compose}) is valid for all $f\in L^2(G)$.
This construction has a handy consequence which we describe next.

For the sake of notational simplicity, if $Z$ is a random field,
then we write $Z\circ h^{-1}$ for the random field $Z_t(h^{-1}(x))$, whenever
$h$ is such that this definition makes sense. Of course, if $Z$ is nonrandom,
then we may use the very same notation; thus, $K\circ h^{-1}$ makes
sense equally
well in what follows.

%
%
%le11.5 #&#
\begin{lemma}\label{lemZcirch}
Let $\beta\ge0$ and $h\in\operatorname{Iso}(G,\Gamma)$.
If $Z\in\mathcal{P}^2_\beta(G)$, then
$Z\circ h^{-1}\in\mathcal{P}^2_\beta(\Gamma)$, where
%
%
%e11.5 #&#
\begin{equation}
\bigl( Z\circ h^{-1} \bigr)_t(x):= Z_t
\bigl(h^{-1}(x) \bigr)\qquad\mbox{for all }t>0\mbox{ and }x\in\Gamma.
\end{equation}
Moreover,
%
%
%e11.6 #&#
\begin{equation}
\mathcal{N}_\beta \bigl(Z\circ h^{-1};\Gamma \bigr)=\sqrt{
\mu(h)} \mathcal{N}_\beta(Z;G).
\end{equation}
\end{lemma}

\begin{pf}
It suffices to prove the lemma when $Z$ is an elementary random field.
But then the result follows immediately from first principles,
thanks to (\ref{eqL2compose}).
\end{pf}

Our next result is a change of variables formula for Wiener integrals.

%
%
%le11.6 #&#
\begin{lemma}\label{lemWIcompose}
If $F\in L^2(\R_+\times\Gamma)$ and $h\in\operatorname{Iso}(G,\Gamma)$,
then
%
%
%e11.7 #&#
\begin{equation}
\int_{\R_+\times G} (F\circ h)\, \d\xi= \frac{1}{\sqrt{\mu(h)}} \int
_{\R_+\times\Gamma} F \,\d\xi_h\qquad\mbox{a.s.}
\end{equation}
\end{lemma}

\begin{pf}
Thanks to the very construction of Wiener integrals,
it suffices to prove the lemma in the case that
$F_t(x) = A\mathbf{1}_{[c,d]}(t)\mathbf{1}_Q(x)$
for some $A\in\R$, $0\le c<d$, and Borel-measurable set
$Q\subset\Gamma$ with $m_\Gamma(Q)<\infty$.
In this special case, $(F\circ h)_t(x) = A\mathbf{1}_{[c,d)}(t)\mathbf
{1}_{h^{-1}(Q)}(x)$,
whence we have
%
%
%e11.8 #&#
\begin{equation}
\int_{\R_+\times G}(F\circ h)\, \d\xi= A\xi \bigl( [c,d )\times
h^{-1}(Q) \bigr)
\end{equation}
which is $[\mu(h)]^{-1/2}$ times
$A\xi_h ( [c,d)\times
Q )=\int_{\R_+\times\Gamma}F \,\d\xi_h$,
by default.
\end{pf}

%
%
%le11.7 #&#
\begin{lemma}\label{lemSCcompose}
Let $\circledast$ denote stochastic convolution with respect to
the white noise $\xi$ on $\R_+\times G$, as before.
For every $h\in\operatorname{Iso}(G,\Gamma)$, let $\circledast_h$
denote stochastic convolution with respect to the white noise $\xi_h$
on $\R_+\times\Gamma$. Choose and fix some $\beta\ge0$.
Then, for all $K\in\mathcal{L}^2_\beta(\Gamma)$ and
$Z\in\mathcal{P}^2_\beta(\Gamma)$,
%
%
%e11.9 #&#
\begin{equation}
\label{eqSCcompose} (K\circ h)\circledast(Z\circ h) = \frac{1}{\sqrt
{\mu(h)}} (K
\circledast_h Z )\circ h,
\end{equation}
almost surely.
\end{lemma}

\begin{pf}
Lemma~\ref{lemmuWN} shows that $\xi_h$ is indeed a white
noise on $\R_+\times\Gamma$; and Lemma~\ref{lemZcirch}
guarantees that $Z\circ h\in\mathcal{P}^2_\beta(G)$. In order
for $(K\circ h)\circledast(Z\circ h)$ to be a well-defined stochastic
convolution, we need $K\circ h$ to be in $\mathcal{L}^2_\beta(G)$
(Theorem~\ref{thSC}). But (\ref{eqL2compose}) tells us that
%
%
%e11.10 #&#
\begin{equation}
\llVert K_t\circ h\rrVert_{L^2(G)}^2 =
\frac{1}{\mu(h)}\llVert K_t\rrVert_{L^2(\Gamma
)}^2
\qquad\mbox{for all }t>0,
\end{equation}
and hence
%
%
%e11.11 #&#
\begin{equation}
\llVert K\circ h\rrVert_{\mathcal{L}^2_\beta(G)}^2 =\frac{1}{\mu(h)}
\llVert K\rrVert_{\mathcal{L}^2_\beta(\Gamma)}^2<\infty.
\end{equation}
This shows that $(K\circ h)\circledast(Z\circ h)$ is a properly-defined
stochastic convolution.

In order to verify (\ref{eqSCcompose}), which is the main content
of the lemma, it suffices to consider the case that $K$ and $Z$ are
both elementary; see Lemma~\ref{lemYoung} and our construction of
stochastic convolutions. In other words, it remains to consider the case
that $K$ and $Z$ have the form described in~(\ref{eqKZ}): that is, in
the present
context: (i) $K_s(y) = A\mathbf{1}_{(c,d]}(s)\phi(y)$ where
$A\in\R$, $0\le c<d$, and $\phi\in C_c(\Gamma)$; and (ii)
$Z_t(x)=X\mathbf{1}_{[a,b)}(t)\psi(x)$ for $0<a<b$, $X\in L^2(\P)$ is
$\F_a$-measurable, and $\psi\in C_c(\Gamma)$. In this case,
\begin{eqnarray}
(K\circ h)_s(y) &=& A\mathbf{1}_{(c,d]}(s)\phi \bigl(h(y)
\bigr),
\nonumber\\[-8pt]\\[-8pt]\nonumber
(Z\circ h)_t(x) &=& X\mathbf{1}_{(a,b]}(t)\psi \bigl(h(x) \bigr).\nonumber
\end{eqnarray}
Therefore,
%
%
%e11.12 #&#
\begin{eqnarray}
\qquad\quad&& \bigl[ (K\circ h)\circledast(Z\circ h) \bigr]_t(x)
\nonumber\\[-8pt]\\[-8pt]\nonumber
&&\qquad = AX\int
_{(0,t)\times G}\mathbf{1}_{(c,d]}(s)\mathbf{1}_{(a,b]}(t-s)
\phi \bigl(h \bigl(yx^{-1} \bigr) \bigr)\psi \bigl(h(y) \bigr) \xi(\d s
\,\d y).
\end{eqnarray}
The preceding integral is a Wiener integral, and the above quantity is almost
surely equal to
%
%
%e11.13 #&#
\begin{eqnarray}
\qquad&& \frac{AX}{\sqrt{\mu(h)}}\int_{(0,t)\times\Gamma} \mathbf{1}_{(c,d]}(s)
\mathbf{1}_{(a,b]}(t-s) \phi \bigl(y \bigl(h(x) \bigr)^{-1}
\bigr) \psi(y) \xi_h(\d s \,\d y)
\nonumber\\[-8pt]\\[-8pt]\nonumber
&&\qquad  = \frac{1}{\sqrt{\mu(h)}}(K
\circledast_h Z)_t \bigl(h(x) \bigr),
\end{eqnarray}
thanks to Lemma~\ref{lemWIcompose}.
\end{pf}

Finally, if $X:=\{X_t\}_{t\ge0}$ is a L\'evy process on $G$, then
$Y_t:=h(X_t)$ defines a L\'evy process $Y:=h\circ X$ on $\Gamma$.
In order to identify better the process $Y:=h\circ X$, let us first
recall \cite{Morris}, Chapter~4,
that since $\Gamma=h(G)$, every character $\zeta\in\Gamma^*$
is of the form $\chi\circ h^{-1}$ for some $\chi\in G^*$ and vice versa.
In particular, we can understand the dynamics of $Y=h\circ X$ via
the following computation:
%
%
%e11.14 #&#
\begin{eqnarray}
\E(\zeta, Y_t) &=& \E \bigl(\chi\circ h^{-1},Y_t
\bigr) = \E \bigl[ \chi \bigl( h^{-1}(Y_t) \bigr) \bigr] =\E \bigl[\chi(X_t) \bigr]
\nonumber\\[-8pt]\\[-8pt]\nonumber
&=& \E(\chi,X_t )= \E(\zeta
\circ h,X_t ),
\end{eqnarray}
for every $t\ge0$ and $\zeta=\chi\circ h^{-1}\in\Gamma^*$.
Let $\Psi_W$ denote the characteristic exponent
of every L\'evy process $W$. Then it follows that
%
%
%e11.15 #&#
\begin{equation}
\Psi_{h\circ X}(\zeta) = \Psi_X (\zeta\circ h )\qquad\mbox{for
all }\zeta\in\Gamma^*.
\end{equation}
In particular, we can evaluate the $\Upsilon$-function for $Y:=h\circ
X$ as
follows:
%
%
%e11.16 #&#
\begin{equation}
\qquad \int_{\Gamma^*} \biggl(\frac{1}{1+\Re\Psi_{h\circ X}(\zeta)} \biggr)
m_{\Gamma^*}(\d\zeta) =\int_{\Gamma^*} \biggl(
\frac{1}{1+\Re\Psi_X(\zeta\circ h)} \biggr) m_{\Gamma^*}(\d\zeta).
\end{equation}
Since $\zeta\circ h$ is identified with $\chi$ through the
Pontryagin--van Kampen duality pairing, we find the
familiar fact that $\Gamma^*\cong G^*$ \cite{Morris}, Chapter~4,
whence we may deduce the following:
%
%
%e11.17 #&#
\begin{eqnarray}
&& \int_{\Gamma^*} \biggl(\frac{1}{1+\Re\Psi_{h\circ X}(\zeta)} \biggr)
m_{\Gamma^*}(\d\zeta)\nonumber
\\
&&\qquad  =\int_{G^*} \biggl(
\frac{1}{1+\Re\Psi_X(\chi)} \biggr) \bigl(m_{\Gamma^*}\circ h^{-1} \bigr) (
\d\chi)
\\
&&\qquad =\mu(h)\cdot\int_{G^*} \biggl(\frac{1}{1+\Re\Psi_X(\chi)} \biggr)
m_{G^*}(\d\chi).\nonumber
\end{eqnarray}
This $\mu(h)$ is the same as the constant in (\ref{eq11.1}), because our
normalization of Haar measures makes the Fourier transform an
$L^2$-isometry.

In other words, we have established the following.

%
%
%le11.8 #&#
\begin{lemma}\label{lemDalang2}
Let $X:=\{X_t\}_{t\ge0}$ denote a L\'evy process on $G$,
and choose and fix $h\in\operatorname{Iso}(G,\Gamma)$.
Then the $G$-valued process
$X$ satisfies Dalang's condition (\ref{D}) if and only if
the $\Gamma$-valued process $Y:=h\circ X$
satisfies Dalang's condition (\ref{D}).
\end{lemma}

Let us make another simple computation, this time about the
invariance properties of semigroups and their $L^2$-generators.

%
%
%le11.9 #&#
\begin{lemma}\label{lemIsoTD}
Let $X:=\{X_t\}_{t\ge0}$ denote a L\'evy process on $G$,
with semigroup $\{P_t^X\}_{t\ge0}$ and generator $\mathscr{L}^X$,
and choose and fix $h\in\operatorname{Iso}(G,\Gamma)$. Then the
semigroup and generator of $Y:=h\circ X$ are
%
%
%e11.18 #&#
\begin{equation}
\bigl(P^{h\circ X}_t f \bigr) (y) = \bigl( P^X_t(f
\circ h) \bigr) \bigl(h^{-1}(y) \bigr)
\end{equation}
and
%
%
%e11.19 #&#
\begin{equation}
\bigl( \mathscr{L}^{h\circ X} f \bigr) (y) = \bigl(\mathscr{L}^X
(f\circ h) \bigr) \bigl(h^{-1}(y) \bigr),
\end{equation}
respectively, where $t\ge0$, $y\in\Gamma$, and $f\in L^2(\Gamma)$.
\end{lemma}

\begin{pf}
If $t\ge0$ and $y\in\Gamma$, then
$yh(X_t)=h ( h^{-1}(y)X_t )$, whence it follows that
for all $f\in C_c(\Gamma)$,
%
%
%e11.20 #&#
\begin{equation}
\bigl( P^{h\circ X}_tf \bigr) (y) = \E \bigl[ f
\bigl(yh(X_t) \bigr) \bigr]= \E \bigl[(f\circ h) \bigl(
h^{-1}(y)X_t \bigr) \bigr].
\end{equation}
This yields the semigroup of $h\circ X$ by the density of $C_c(G)$
in $L^2(G)$.
Differentiate with respect to $t$ to compute the generator.
\end{pf}

As a ready consequence of Lemma~\ref{lemIsoTD}, we find that if
$X:=\{X_t\}_{t\ge0}$ denotes a L\'evy process on $G$ with transition
densities $p^X$ (with respect to $m_G$), and if $h\in\operatorname
{Iso}(G,\Gamma)$,
then $h\circ X$ is a L\'evy process on $\Gamma$ with transition
densities $p^{h\circ X}$ (with respect to $m_\Gamma$) that are given by
%
%
%e11.21 #&#
\begin{equation}
\label{eqphcircX} p^{h\circ X}:= \frac{p^X\circ h^{-1}}{\mu(h)}.
\end{equation}
Indeed, Lemma~\ref{lemIsoTD} and the definition of $\mu(h)$
together imply that
%
%
%e11.22 #&#
\begin{equation}
\int\psi p_t^{h\circ X} \,\d m_\Gamma=\E \bigl[\psi
\bigl(h(X_t) \bigr) \bigr],
\end{equation}
for all $t>0$ and $\psi\in C_c(G)$. Therefore, $p^{h\circ X}$ is a
version of
the transition density of $h\circ X$. Lemma~\ref{lemTrDensity}
ensures that $p^{h\circ X}$ is in fact the unique continuous version
of any such transition density.

We are ready to present and prove the main result of this section.
Throughout, $X:=\{X_t\}_{t\ge0}$ denotes a L\'evy process
on $G$ that satisfies Dalang's condition (\ref{D}), and recall our convention
that either $G$ is compact or $\sigma(0)=0$. In this way, we see that
(\ref{SHE}) has a unique solution for every nonrandom initial function
in $L^2(G)$.

%
%
%th11.10 #&#
\begin{theorem}[(Group invariance of SPDEs)]\label{thgroupinv}
Suppose $u_0\in L^2(G)$ is nonrandom, and let $u$
denote the solution to (\ref{SHE})---viewed as
an SPDE on $(0,\infty)\times G$---whose
existence and uniqueness is guaranteed\vspace*{1pt} by Theorem
\ref{thexistunique}. Choose and fix $h\in\operatorname{Iso}(G,\Gamma)$.
Then $v_t:= u_t\circ h^{-1}$ defines the unique solution to the stochastic
heat equation
%
%
%e11.23 #&#
\begin{equation}
\left|
\begin{array} {l} \displaystyle
\frac{\partial v_t(x)}{\partial t} = \bigl(\mathscr{L}^{h\circ X} v_t \bigr) (x)
+\lambda\sqrt{\mu(h)} \sigma \bigl( v_t(x) \bigr)\xi_h,
\\[8pt]
\displaystyle v_0= u_0\circ h^{-1},
\end{array} %
\right.\label{eqv}
\end{equation}
viewed as an SPDE on $\Gamma=h(G)$, for $x\in\Gamma$ and $t>0$.
\end{theorem}

\begin{pf}
With the groundwork under way, the proof is quite simple.
Let $v$ be the solution to (\ref{eqv}); its existence is guaranteed
thanks to Lemma~\ref{lemDalang2} and Theorem~\ref{thexistunique}.

Let $v^{(n)}$ and $u^{(n)}$, respectively, denote the Picard
iterates of (\ref{eqv}) and $u$. That is, $u^{(n)}$'s are defined iteratively
by (\ref{equ^n}), and $v$'s are defined similarly as
%
%
%e11.24 #&#
\begin{equation}
v^{(n+1)}_t:= P_t^{h\circ X}
v_0+ \lambda\sqrt{\mu(h)} \bigl( p^{h\circ X}
\circledast_h \sigma \bigl(v^{(n)} \bigr)
\bigr)_t.
\end{equation}
We first claim that for all $t>0$,
%
%
%e11.25 #&#
\begin{equation}
\label{eqinductionstep} v^{(n)}_t=u^{(n)}_t
\circ h^{-1}\qquad\mbox{a.s. for all }n\ge0.
\end{equation}
This is a tautology when $n=0$, by construction. Suppose $v^{(n)}_t
=u^{(n)}_t\circ h^{-1}$ a.s. for every $t>0$, where $n\ge0$ is an
arbitrary fixed integer.
We next verify that $v^{(n+1)}_t=u^{(n+1)}_t\circ h^{-1}$ a.s. for all
$t>0$, as well. This and a relabeling $[n\leftrightarrow n+1$]
will establish (\ref{eqinductionstep}).

Thanks to the induction hypothesis, Lemma~\ref{lemIsoTD}
and (\ref{eqphcircX}),
%
%
%e11.26 #&#
\begin{equation}
\qquad v^{(n+1)}_t:= \bigl( P_t^X
u_0 \bigr)\circ h^{-1}+\frac{\lambda}{\sqrt{\mu(h)}} \bigl( \bigl(
p^X\circ h^{-1} \bigr)\circledast_h \sigma
\bigl(u^{(n)}\circ h^{-1} \bigr) \bigr)_t,
\end{equation}
almost surely.
Therefore, Lemma~\ref{lemSCcompose} implies that
%
%
%e11.27 #&#
\begin{equation}
v^{(n+1)}_t:= \bigl( P_t^X
u_0 \bigr)\circ h^{-1}+ \lambda \bigl( p^X
\circledast\sigma \bigl(u^{(n)} \bigr) \bigr)_t\circ
h^{-1},
\end{equation}
almost surely. We now merely recognize the right-hand side as
$u^{(n+1)}_t \circ h^{-1}$; see~(\ref{equ^n}). In this way, we have proved
(\ref{eqinductionstep}).\vspace*{2pt}

Since we now know that
$v^{(n)}=u^{(n)}\circ h^{-1}$, two appeals to Theorem~\ref{thexistunique}
(via Lemma~\ref{lemZcirch}) show that if $\beta$ is sufficiently
large, then
$v^{(n)}$ converges in $\mathcal{P}^2_\beta(\Gamma)$
to $v$ and $u^{(n)}\to u$ in $\mathcal{P}^2_\beta(G)$, as $n\to
\infty$.
Thus, it follows from a second application of Lemma~\ref{lemZcirch}
that $v=u\circ h^{-1}$.
\end{pf}

The following is a ready corollary of Theorem~\ref{thgroupinv};
its main content is in the last line where it shows that our
noise excitation indices are ``invariant under group isomorphisms.''

%
%
%co11.11 #&#
\begin{corollary}\label{coinv}
In the context of Theorem~\ref{thgroupinv},
we have the following energy identity:
%
%
%e11.28 #&#
\begin{equation}
\label{eqinv} \E \bigl(\llVert u_t\rrVert_{L^2(G)}^2
\bigr) = \frac{1}{\mu(h)}\E \bigl( \llVert v_t\rrVert
_{L^2(\Gamma)}^2 \bigr),
\end{equation}
valid for all $t\ge0$. In particular, $u$ and $v$ have the same
noise excitation indices.
\end{corollary}

\begin{pf}
Since $v_t(x)=u_t(h^{-1}(x))$, it follows from Theorem~\ref{thgroupinv}
and (\ref{eqL2compose}) that
%
%
%e11.29 #&#
\begin{equation}
\llVert u_t\rrVert_{L^2(G)}^2 =
\frac{1}{\mu(h)} \llVert v_t\rrVert_{L^2(\Gamma)}^2
\qquad\mbox{a.s.},
\end{equation}
which is more than enough to imply (\ref{eqinv}). The upper noise-excitation
index of $u$ at time $t\ge0$ is
%
%
%e11.30 #&#
\begin{equation}
\label{eqUBeu} \overline{\mathfrak{e}}(t)= \limsup_{\lambda\uparrow
\infty}
\frac{1}{\log\lambda}\log\log\sqrt{\E \bigl(\llVert u_t\rrVert
_{L^2(G)}^2 \bigr)},
\end{equation}
whereas the upper noise excitation index of $v$ at time $t$ is
%
%
%e11.31 #&#
\begin{equation}
\label{eqUBev} \limsup_{\lambda\uparrow\infty}\frac{1}{\log[
\lambda\sqrt{\mu(h)} ]
}\log\log\sqrt
{\E \bigl(\llVert v_t\rrVert_{L^2(\Gamma)}^2
\bigr)},
\end{equation}
which is equal to $\overline{\mathfrak{e}}(t)$,
thanks to (\ref{eqinv}) and the fact that $\log[
\lambda\sqrt{\mu(h)}]
\sim\log\lambda$ as $\lambda\uparrow\infty$. This proves that the
upper excitation indices of $u$ and $v$ are the same. The very same proof
shows also that the lower excitation indices are shared as well.
\end{pf}

%s12 #&#
\section{Projections}\label{secprojections}

Consider our stochastic heat equation (\ref{SHE})
in the case that the underlying LCA group $G$ is noncompact, metrizable and
has more than one element; that is, consider the general setting of
Theorem~\ref{thconnected1}.
According to the structure theory of LCA groups, which we will recall
in due time,
we can write $G\cong\R^n\times K$ for a nonnegative integer
$n$ and a compact LCA group $K$. It is easy to see that
the underlying L\'evy process on $G$ can then be written---coordinatewise---as
$X\times Y:=\{X_t\times Y_t\}_{t\ge0}$,
where $\{X_t\}_{t\ge0}$ is a L\'evy process on $\R^n$
and $\{Y_t\}_{t\ge0}$ a L\'evy process on $K$.
The results of this section will allow us
to compare the energy of our stochastic PDE to the energy of another
version of (\ref{SHE}), whose
$x$-variable now ranges in $\R^n$, and whose operator $\mathscr{L}$
is the generator of
$\{X_t\}_{t\ge0}$. This comparison principle is a kind of parallel to
the classical
energy inequality of potential theory. In the present setting, it
states that the energy of
(\ref{SHE}) on $G\cong\R^n\times K$---using the L\'evy process $X\times
Y$---is greater than
or equal to the energy of (\ref{SHE}) on $\R^n$---using the L\'evy process
$X$. Moreoever,
if (\ref{SHE}) has a solution---that is, if $X\times Y$ satisfies Dalang's
condition (\ref{D})---then
(\ref{SHE}) on $\R^n$ must have a solution---that is, $X$ must satisfy Dalang's
condition (\ref{D})---and hence $n=1$.
The structure theory of L\'evy processes on $\R$ will then show us
that the lowest
energy we can expect is from the case that $X$ is Brownian motion. In
that case,
a simple scaling argument can yield the desired $\exp\{\operatorname
{const}\cdot\,\lambda^4\}$
lower bound, which will ultimately verify Theorem~\ref{thconnected1}.

In this section, we study the natural projection $G$ of a (larger) LCA
group $G\times K$,
where $K$ is a compact Abelian group. It is easy to see from first
principles that
such a projection maps a L\'evy process on $G\times K$ to a L\'evy
process on~$G$.
One of the main results of this section is
that if the original process on $G$ satisfied Dalang's condition
(\ref{D})---on $G\times K$---%
then the new process on $G$ will satisfy condition (\ref{D}) on $G$. Thanks
to the structure theory of LCA groups,
this fact and its ensuing ``energy inequality'' will be instrumental in
the proof of
Theorem~\ref{thconnected1} (see Section~\ref{secproofs}).

We will prove Theorem~\ref{thconnected1} in Section~\ref{secproofs}.
Presently, we satisfy ourselves by stating and proving
a general form of the mentioned projection
theorem/energy inequality.

Throughout this section, we let
$G$ denote an LCA group and $K$ a compact Abelian group.
Then it is well known, and easy to see directly, that
$G\times K$ is an LCA group with dual group $(G\times K)^*=G^*\times K^*$
\cite{Morris}, Chapter~4. (For purposes of comparison, let us state
that the $G\times K$
of this section is going to play the role of $G\cong\R^n\times K$ of
the preceding
paragraphs.)

Let $\pi\dvtx G\times K\to G$ denote the canonical projection map.
Since $\pi$ is a (continuous) group homomorphism, it follows that if
$X:=\{X_t\}_{t\ge0}$ is a L\'evy process on $G\times K$,
then $(\pi\circ X)_t:= \pi(X_t)$ defines a L\'evy process on $G$.
If $\chi\in G^*$, then $\chi\circ\pi\in(G\times K)^*$, and hence
%
%
%e12.1 #&#
\begin{equation}
\E \bigl(\chi,\pi(X_t) \bigr) =\E \bigl[ ( \chi\circ\pi,
X_t ) \bigr] = {\e}^{-t\Psi_X(\chi\circ\pi)},
\end{equation}
for all $t\ge0$ and $\chi\in G^*$. In other words, we can write the
characteristic
exponent of $\pi\circ X$ in terms of the characteristic exponent of $X$
as follows:
%
%
%e12.2 #&#
\begin{equation}
\Psi_{\pi\circ X}(\chi) = \Psi_X(\chi\circ\pi)\qquad\mbox{for all
} \chi\in G^*.
\end{equation}

%
%
%pr12.1 #&#
\begin{proposition}\label{prproj1}
If $X$ satisfies Dalang's condition (\ref{D}) on $G\times K$, then the L\'evy process
$\pi\circ X$ satisfies condition (\ref{D}) on $G$. In fact, we have
the following ``energy inequality'':
%
%
%e12.3 #&#
\begin{equation}
\Upsilon_{\pi\circ X}(\beta) \le\Upsilon_X(\beta)\qquad\mbox{for
all }\beta\ge0,
\end{equation}
where $\Upsilon_W$ is the function defined in~(\ref{eqUpsilon})
and/or (\ref{eqUpsilon1}) for every L\'evy process $W$ that has
transition densities.
\end{proposition}

\begin{pf}
First of all, note that the product measure $m_G\times m_K$
is a translation-invariant Borel measure on $G\times K$, whence
$m_{G\times K}=cm_G\times m_K$ for some constant~$c$.
It is easy to see that $c\in(0,\infty)$; let us argue next that $c=1$.
If $f\in L^2(G)$ and $g\in L^2(K)$ satisfy
$m_G\{f>0\}>0$ and $m_K\{g>0\}>0$, then $(f\otimes g)(x\times y):=
f(x)g(y)$ satisfies $f\otimes g\in L^2(G\times K)$, and
%
%
%e12.4 #&#
\begin{equation}
\llVert f\otimes g\rrVert_{L^2(G\times K)}=\llVert f\rrVert_{L^2(G)}
\llVert g\rrVert_{L^2(K)} =\llVert f\otimes g\rrVert_{L^2(m_G\times m_K)}.
\end{equation}
Since the left-most term is equal to $c$ times the right-most term, it
follows that
$c=1$.

Let $p^W$ denote the transition densities of $W$ for every L\'evy process
$W$ that possesses transition densities.
It is a simple fact about ``marginal probability densities'' that since
$X$ has nice transition densities $p^X$ (see Lemma~\ref{lemTrDensity}),
so does $\pi\circ X$. In fact, because $m_{G\times K}
=m_G\times m_K$---as was proved in the previous paragraph---we may
deduce that
%
%
%e12.5 #&#
\begin{equation}
p^{\pi\circ X}_t(x) = \int_K
p^X_t(x\times y) m_K(\d y)\qquad\mbox{for
all }t>0\mbox{ and }x\in G.
\end{equation}
Now we simply compute: because $K$ is compact, $m_K$ is a probability measure,
and hence
%
%
%e12.6 #&#
%e12.7 #&#
\begin{eqnarray}
\bigl\llVert p_t^X \bigr\rrVert_{L^2(G\times K)}^2
&=& \int_G m_G(\d x)\int_K
m_K(\d y) \bigl\llvert p_t^X(x\times y)
\bigr\rrvert^2
\\
\nonumber
&\ge&\int_G m_G(\d x) \biggl\llvert
\int_K m_K(\d y) p_t^X(x
\times y) \biggr\rrvert^2
\\
&=& \bigl\llVert p^{\pi\circ X}_t \bigr\rrVert_{L^2(G)}^2,
\end{eqnarray}
for all $t>0$, owing to the Cauchy--Schwarz inequality.
We can integrate both sides of the preceding $[\exp(-\beta t)\, \d t]$ in
order to see that
%
%
%e12.8 #&#
\begin{equation}
\int_0^\infty{\e}^{-\beta s} \bigl\llVert
p_s^{\pi\circ X} \bigr\rrVert_{L^2(G)}^2 \,\d s
\le\int_0^\infty{\e}^{-\beta s} \bigl\llVert
p_s^X \bigr\rrVert_{L^2(G\times K)}^2 \,\d s,
\end{equation}
for all $\beta\ge0$, and the result follows.
\end{pf}

%s13 #&#
\section{An abstract lower bound}\label{seclowerbound}

The main result of this section is an abstract lower estimate for
the energy of the solution in terms of the function $\Upsilon$
that was defined in~(\ref{eqUpsilon}); see also (\ref{eqUpsilon1}).

%
%
%pr13.1 #&#
\begin{proposition}\label{prL2LB}
If $u_0\in L^2(G)$, $\llVert u_0\rrVert _{L^2(G)}>0$, and (\ref{eqLsigma})
holds, then
there exists a finite constant $c\ge1$ such that
%
%
%e13.1 #&#
\begin{equation}
\mathscr{E}_t(\lambda) \ge c^{-1}\exp(-ct)\cdot\sqrt
{ 1+\sum_{j=1}^\infty \biggl(
\frac{\ell_\sigma^2\lambda^2}{{\e}} \cdot\Upsilon(j/t) \biggr)^j},
\end{equation}
for all $t\ge0$. The constant $c$ depends on $u_0$ as well as the
underlying L\'evy process~$X$.
\end{proposition}

\begin{pf}
Consider first the case that
%
%
%e13.2 #&#
\begin{equation}
\label{equ0LinftyL2} u_0\in L^\infty(G)\cap L^2(G).
\end{equation}
Thanks to (\ref{equ0LinftyL2}), we may apply (\ref{eqisometry});
upon integration [$m_G(\d x)]$, this and Fubini's theorem
together yield the following formula:
%
%
%e13.3 #&#
\begin{eqnarray}\label{eqL2always}
\qquad \E \bigl(\llVert u_t\rrVert_{L^2(G)}^2 \bigr)&=&
\llVert P_tu_0\rrVert_{L^2(G)}^2 +
\lambda^2\int_0^t \llVert
p_{t-s}\rrVert_{L^2(G)}^2 \E \bigl(\llVert\sigma
\circ u_s\rrVert_{L^2(G)}^2 \bigr)\, \d s\nonumber
\\
&\ge&\llVert P_tu_0\rrVert_{L^2(G)}^2
+ \ell_\sigma^2\lambda^2\int
_0^t \llVert p_{t-s}\rrVert
_{L^2(G)}^2 \E \bigl(\llVert u_s\rrVert
_{L^2(G)}^2 \bigr)\, \d s
\\
&=& \llVert P_tu_0\rrVert_{L^2(G)}^2
+ \ell_\sigma^2\lambda^2\int
_0^t \bar{p}_{t-s}({
\e}_G) \E \bigl(\llVert u_s\rrVert_{L^2(G)}^2
\bigr)\, \d s.\nonumber
\end{eqnarray}
Appeals to Fubini's theorem are indeed justified, since Theorem~\ref
{thexistunique}
contains implicitly the desired measurability statements about $u$.

Next we prove that (\ref{eqL2always}) holds for every $u_0\in L^2(G)$
and not just those that satisfy (\ref{equ0LinftyL2}). With\vspace*{1pt} this aim
in mind,
let us appeal to density in order to
find $u_0^{(1)},u_0^{(2)},\ldots\in L^\infty(G)\cap L^2(G)$ such that
%
%
%e13.4 #&#
\begin{equation}
\lim_{n\to\infty} \bigl\llVert u^{(n)}_0-u_0
\bigr\rrVert_{L^2(G)}=0.
\end{equation}
Then (\ref{equ-v}) assures us
that there exists $\beta>0$, sufficiently large, such that
%
%
%e13.5 #&#
\begin{equation}
\label{equ-v1} \lim_{n\to\infty}\mathcal{N}_\beta \bigl(
u^{(n)}-u \bigr)=0,
\end{equation}
where $u^{(n)}_t(x)$ denotes the solution to (\ref{SHE}) with initial value
$u^{(n)}_0$. Equation (\ref{equ-v1}) implies readily that
%
%
%e13.6 #&#
\begin{equation}
\lim_{n\to\infty}\E \bigl( \bigl\llVert u^{(n)}_t
\bigr\rrVert_{L^2(G)}^2 \bigr) =\E \bigl(\llVert
u_t\rrVert_{L^2(G)}^2 \bigr)\qquad\mbox{for all }
t\ge0.
\end{equation}
And because $P_t$ is contractive on $L^2(G)$,
%
%
%e13.7 #&#
\begin{equation}
\lim_{n\to\infty} \bigl\llVert P_tu^{(n)}_0
\bigr\rrVert_{L^2(G)}= \llVert P_tu_0\rrVert
_{L^2(G)}\qquad\mbox{for all }t\ge0.
\end{equation}
Therefore, our claim that (\ref{eqL2always}) holds is verified once
we show that
%
%
%e13.8 #&#
\begin{equation}
\label{eqgoal} \lim_{n\to\infty}\int_0^t
\bar{p}_{t-s}({\e}_G) \E \bigl( \bigl\llVert
u^{(n)}_s - u_s \bigr\rrVert
_{L^2(G)}^2 \bigr)\, \d s=0
\end{equation}
for every $t>0$.
This is so because of (\ref{equ-v1}) and
the fact that the preceding integral is bounded above by
%
%
%e13.9 #&#
\begin{eqnarray}
&& \bigl[ \mathcal{N}_\beta \bigl(u^{(n)}-u \bigr)
\bigr]^2\cdot\int_0^t{
\e}^{-2\beta(t-s)}\bar{p}_{t-s}({\e}_G)\, \d s
\nonumber\\[-8pt]\\[-8pt]
&&\qquad \le \bigl[
\mathcal{N}_\beta \bigl(u^{(n)}-u \bigr) \bigr]^2
\cdot\Upsilon(2\beta);\nonumber
\end{eqnarray}
see also (\ref{eqUpsilon1}). Thus, we have established (\ref{eqL2always})
in all cases of interest. We can now proceed to prove the main part of the
proposition.

Let us define, for all $t>0$,\footnote{It is easy to write
$\mathcal{E}$ in terms of the energy of
the solution. Indeed, $\mathcal{E}(t)=[\mathscr{E}_t(\lambda)]^2$.}
%
%
%e13.10 #&#
\begin{eqnarray}
\mathcal{P}(t) &:=& \ell_\sigma^2\lambda^2
\bar{p}_t({\e}_G),\qquad \mathcal{I}(t):= \llVert
P_tu_0\rrVert_{L^2(G)}^2,
\nonumber\\[-8pt]\\[-8pt]
\mathcal{E}(t) &:=& \E \bigl(\llVert u_t\rrVert_{L^2(G)}^2\bigr).\nonumber
\end{eqnarray}
Thanks to (\ref{eqL2always}),
we obtain the pointwise convolution inequality
%
%
%e13.11 #&#
\begin{eqnarray}\label{E}
\mathcal{E} &\ge&\mathcal{I} + (\mathcal{P}*\mathcal{E})\nonumber
\\
&\ge&\mathcal{I} + (\mathcal{P}*\mathcal{I}) + (\mathcal{P}*\mathcal{P}*\mathcal{E})
\nonumber\\[-8pt]\\[-8pt]\nonumber
& \vdots&
\\
&\ge&\mathcal{I} + (\mathcal{P}*\mathcal{I}) +(\mathcal{P}*\mathcal{P}*
\mathcal{I}) + (\mathcal{P}*\mathcal{P}*\mathcal{P}*\mathcal{I})+\cdots,\nonumber
\end{eqnarray}
where $(\psi*\phi)(t):=\int_0^t\psi(s)\phi(t-s)\, \d s$
defines the usual (temporal) convolution operator ``$*$.'' In
particular, we may note
that the final quantity in~(\ref{E}) depends only on the function
$\mathcal{I}$, which is
related only to the initial function $u_0$.

A direct computation shows us that
the Fourier transform of $P_tu_0$, evaluated at $\chi\in G^*$, is
$\exp\{-t\Psi(\chi^{-1})\}\hat{u}_0(\chi)$; see (\ref{eqinversion}).
Therefore, we may apply the Plancherel's theorem to see that
%
%
%e13.12 #&#
\begin{equation}
\mathcal{I}(t) = \int_{G^*}{\e}^{-2t\Re\Psi(\chi)} \bigl\llvert
\hat{u}_0(\chi) \bigr\rrvert^2 m_{G^*}( \d\chi)
\qquad\mbox{for all }t>0.
\end{equation}
Since $u_0\in L^2(G)$, we can find a compact neighborhood $K$
of the identity of $G^*$ such that
%
%
%e13.13 #&#
\begin{equation}
\qquad\int_K \bigl\llvert\hat{u}_0(\chi) \bigr
\rrvert^2 m_{G^*}(\d\chi)\ge\frac{1}2 \int
_{G^*} \bigl\llvert\hat{u}_0(\chi) \bigr\rrvert
^2 m_{G^*}(\d\chi)=\frac{1}2\llVert
u_0\rrVert_{L^2(G)}^2,
\end{equation}
thanks to Plancherel's theorem (as well as the monotone convergence
theorem, of course). In this way, we find that
%
%
%e13.14 #&#
\begin{equation}
\label{eqILB} \mathcal{I}(t) \ge\frac{\llVert u_0\rrVert
_{L^2(G)}^2}{2} {\e}^{-c_0t} \qquad
\mbox{for all }t>0,
\end{equation}
where
%
%
%e13.15 #&#
\begin{equation}
c_0:= 2\sup_{\chi\in K}\Re\Psi(\chi).
\end{equation}
We will require the fact that $0\le c_0<\infty$; this fact holds simply
because $\Psi$ is continuous and $\Re\Psi$ is nonnegative.
In this way, (\ref{eqILB}) yields an estimate for
the first term on the right-hand side of (\ref{E}).

As for the other terms, let us write $\mathcal{P}^{*(n)}$ in place
of the $n$-fold convolution, $\mathcal{P}*\cdots*\mathcal{P}$,
where $\mathcal{P}^{*(1)}:=\mathcal{P}$. Then it is easy to deduce
from (\ref{eqILB}) that
%
%
%e13.16 #&#
\begin{equation}
\qquad \bigl( \mathcal{P}^{*(n)}*\mathcal{I} \bigr) (t) \ge\frac{\llVert
u_0\rrVert _{L^2(G)}^2}{2}{
\e}^{-c_0t} \bigl(\mathcal{P}^{*(n)}*\1 \bigr) (t)\qquad \mbox{for
all } t>0,
\end{equation}
where $\1(t):=1$ for all $t>0$. Thus, we conclude from (\ref{E}) that
%
%
%e13.17 #&#
\begin{equation}
\label{RE} \mathcal{E}(t) \ge\frac{\llVert u_0\rrVert
_{L^2(G)}^2}{2}{\e}^{-c_0t} \cdot
\sum_{n=0}^\infty \bigl(\mathcal{P}^{*(n)}*
\1 \bigr) (t),
\end{equation}
where $\mathcal{P}^{*(0)}*\1:=1$.

Now,
%
%
%e13.18 #&#
\begin{equation}
(\mathcal{P}*\1) (t) = \ell_\sigma^2\lambda^2
\cdot\int_0^t\bar{p}_s({
\e}_G)\, \d s.
\end{equation}
Consequently,
%
%
%e13.19 #&#
\begin{eqnarray}
\quad\qquad&& (\mathcal{P}*\mathcal{P}*\1) (t)\nonumber
\\
&&\qquad = \ell_\sigma^4
\lambda^4 \cdot\int_0^t
\bar{p}_{s_2}({\e}_G)\, \d s_2\int
_0^{t-s_2} \bar{p}_{s_1}({
\e}_G)\, \d s_1,\nonumber
\\
&& (\mathcal{P}*\mathcal{P}*\mathcal{P}*\1) (t)
\\
&&\qquad =\ell_\sigma^8
\lambda^8\cdot\int_0^t
\bar{p}_{s_3}({\e}_G)\, \d s_3\int
_0^{t-s_3} \bar{p}_{s_2}({
\e}_G)\, \d s_2 \int_0^{t-s_3-s_2}
\bar{p}_{s_1}({\e}_G)\, \d s_1,\nonumber
\\
&&\hspace*{25pt} \vdots.\nonumber
\end{eqnarray}
For all real $t\ge0$ and integers $n\ge1$,
%
%
%e13.20 #&#
\begin{eqnarray}
\bigl( \mathcal{P}^{*(n)}*\1 \bigr) (t) &\ge&\ell_\sigma
^{2n}\lambda^{2n} \biggl(\int_0^{t/n}
\bar{p}_s({\e}_G)\, \d s \biggr)^n
\nonumber\\[-8pt]\\[-8pt]\nonumber
&\ge& \biggl(\frac{\ell_\sigma^2\lambda^2}{{\e}} \cdot\Upsilon(n/t) \biggr)^n.
\end{eqnarray}
The first bound follows from an application of induction
to the variable $n$, and the second follows
from (\ref{eqpUpsilon}). Since $(\mathcal{P}^{*(0)}*\1)(t)=1$,
the proposition follows from~(\ref{RE}).
\end{pf}

%s14 #&#
\section{Proofs of the main theorems}\label{secproofs}

We have set in place all but one essential ingredients of our proofs. The
remaining part is the following simple real-variable result. We prove the
result in detail, since we will need the following quantitative
form of the ensuing estimates.

%
%
%le14.1 #&#
\begin{lemma}\label{lemsums}
For all integers $a\geq0$ and real numbers $\rho>0$,
there exists a positive and finite constant $c_{a,\rho}>1$
such that
%
%
%e14.1 #&#
\begin{equation}
\sum_{j=a}^\infty \biggl(\frac{b}{j^\rho}
\biggr)^j \ge c^{-1}_{a,\rho}\exp \bigl( (\rho/{\e})
b^{1/\rho} \bigr)\qquad\mbox{for all }b\ge c_{a,\rho}.
\end{equation}
\end{lemma}

\begin{pf}
It is an elementary fact that
$(j/{\e})^j\le j!$ for every integer $j\ge1$. Therefore,
whenever $n$, $m$ and $jm/n$ are positive integers,
%
%
%e14.2 #&#
\begin{equation}
\biggl( \frac{jm}{{\e}n} \biggr)^{jm/n} \le \biggl( \frac{jm}{n}
\biggr)!.
\end{equation}
In particular, for all $b>0$,
%
%
%e14.3 #&#
\begin{equation}
\sum_{j=a}^\infty \biggl( \frac{b}{j^{m/n}}
\biggr)^j \ge\mathop{\sum_{j\ge a}}_{jm\in n\mathbf{Z}_+}
\frac{b^j (m/{\e}n)^{jm/n}}{
(jm/n)!} \ge\mathop{\sum_{k\ge am/n}}_{k\in\mathbf{Z}_+}
\frac{c^k}{k!},
\end{equation}
where $c:= b^{n/m}m/({\e}n)$. Since
%
%
%e14.4 #&#
\begin{equation}
\mathop{\sum_{k<am/n}}_{k\in\mathbf{Z}_+}
\frac{c^k}{k!} \le\max \bigl(b^a,1 \bigr) \sum
_{k=0}^\infty\frac{(m/ {\e}n)^k}{k!} = \exp \biggl\{
\frac{m}{{\e}n} \biggr\}\cdot\max \bigl(b^a,1 \bigr),
\end{equation}
we immediately obtain the inequality
%
%
%e14.5 #&#
\begin{eqnarray}
\sum_{j=a}^\infty \biggl( \frac{b}{j^{m/n}}
\biggr)^j &\ge& {\e}^c - \exp \biggl\{ \frac{m}{{\e}n}
\biggr\}\cdot\max \bigl(b^a,1 \bigr)
\nonumber\\[-8pt]\\[-8pt]\nonumber
&=& \exp \biggl\{
\frac{b^{n/m} m}{{\e}n} \biggr\}- \exp \biggl\{ \frac
{m}{{\e}n} \biggr\} \cdot\max
\bigl(b^a,1 \bigr).
\end{eqnarray}
The preceding bound is valid for all integers $n$ and $m$ that are
strictly positive. We can choose now a sequence $n_k$ and $m_k$
of positive integers such that $\lim_{k\to\infty}(m_k/n_k)=\rho$.
Apply the preceding with $(m,n)$ replaced by $(m_k,n_k)$ and
then let $k\to\infty$ to deduce the following bound:
%
%
%e14.6 #&#
\begin{equation}
\sum_{j=a}^\infty \biggl(\frac{b}{j^\rho}
\biggr)^j \ge\exp \bigl( (\rho/{\e}) b^{1/\rho} \bigr) - \exp (
\rho/{\e} )\cdot\max \bigl(b^a,1 \bigr).
\end{equation}
Since the preceding is valid for all $b>0$, the lemma follows readily.
\end{pf}

With the preceding under way, we conclude the paper by proving
Theorems~\ref{thdiscrete},~\ref{thconnected1} and~\ref{thconnected2} in this
order.

\begin{pf*}{Proof of Theorem~\ref{thdiscrete}}
We plan to appeal to (\ref{eqNu}) in order to verify the stated
energy upper bound.

Since $\Re\Psi$ is nonnegative,
%
%
%e14.7 #&#
\begin{equation}
\Upsilon(\beta) \le\beta^{-1}\qquad\mbox{for all }\beta>0,
\end{equation}
and hence for every $\varepsilon\in(0,1)$,
%
%
%e14.8 #&#
\begin{equation}
\Upsilon^{-1} \biggl( \frac{1}{(1+\varepsilon)^2\lambda^2\lip
^2} \biggr) \le
\operatorname{const}\cdot\,\lambda^2\qquad\mbox{for all }\lambda>1,
\end{equation}
where the implied constant is independent of $\lambda$.
Now we merely apply (\ref{eqNu})
in order to see that there exist finite constants $a$ and $b$
such that $\mathscr{E}_t(\lambda)\le a\exp(b\lambda^2)$ for all
$\lambda>1$.
This proves that $\overline{\mathfrak{e}}(t)\le2$.

For the converse bound, we recall that $m_{G^*}$ has total mass
one because $G^*$ is compact. Since $\Psi$ is continuous, it follows that
$\Re\Psi$ is bounded uniformly on $G^*$, and hence for all $\beta_0>0$
there exists a positive constant such that
%
%
%e14.9 #&#
\begin{eqnarray}
\qquad \Upsilon(\beta) &=&\int_{G^*} \biggl(\frac{1}{\beta+\Re\Psi(\chi
)} \biggr)
m_{G^*}(\d\chi) \ge\frac{\operatorname{const}}{\beta} \qquad\mbox{for all }\beta>
\beta_0.
\end{eqnarray}
Proposition~\ref{prL2LB} then ensures that
%
%
%e14.10 #&#
\begin{equation}
\mathscr{E}_t(\lambda)\ge\operatorname{const} \cdot\,\sqrt{1
+ \sum_{j=1}^\infty \biggl(\frac{t\ell_\sigma^2\lambda^2}{{\e}j}
\biggr)^j} \ge a\exp \bigl(b\lambda^2 \bigr),
\end{equation}
for some finite $a$ and $b$ that depend only on $t$, and
in particular are independent of $\lambda>c_{1,1}$.
(We have appealed to Lemma~\ref{lemsums}---with
$\rho:=1$ and $a:=1$---in order to see that $c_{1,1}$ is strictly
greater than one;
we have also used the assumption that $\ell_\sigma>0$.)
This proves that $\underline{\mathfrak{e}}(t)\ge2$ when
$\ell_\sigma>0$,
and completes our proof of Theorem~\ref{thdiscrete}.
\end{pf*}

\begin{pf*}{Proof of Theorem~\ref{thconnected1}}
First, we consider the case that $G$ is noncompact.

According to the structure theory of LCA groups (\cite{Morris}, Chapter~6),
since $G$ is connected we can find an integer $n\ge0$ and a compact
Abelian group $K$ such that
%
%
%e14.11 #&#
\begin{equation}
G\cong\R^n\times K.
\end{equation}
Because $G$ is not compact, we must have $n\ge1$. Now we put
forth the following claim:
%
%
%e14.12 #&#
\begin{equation}
\label{eqClaim} n=1.
\end{equation}

In order to prove (\ref{eqClaim}), let $\pi$
denote the canonical projection from $G\cong\R^n\times K$ to $\R^n$.
Because condition (\ref{D}) holds for the L\'evy process $X$ on $G\cong\R
^n\times K$,
Proposition~\ref{prproj1} assures us that the L\'evy process
$\pi\circ X$ on $\R^n$ also satisfies condition (\ref{D}).
That is, $\Upsilon_{\pi\circ X}(\beta)<\infty$ for one, hence all,
$\beta>0$. Recall from (\ref{eqUpsilon1}) that
%
%
%e14.13 #&#
\begin{equation}
\qquad \Upsilon_{\pi\circ X}(\beta) = \operatorname{const}\cdot\int
_{\R^n} \biggl(\frac{1}{\beta+\Re\Psi_{\pi\circ X}(z)} \biggr)\, \d z\qquad\mbox{for
all }\beta>0,
\end{equation}
where ``const'' accounts for a suitable normalization of Haar measure on
$\R^n$. Since $\pi\circ X$ is a L\'evy process on $\R^n$,
a theorem of Bochner (\cite{Bochner}, see (3.4.14) on page~67)
ensures that there exists $A\in(0,\infty)$
such that
%
%
%e14.14 #&#
\begin{equation}
\label{eqBochner} \Re\Psi_{\pi\circ X}(z) \le A \bigl(1+\llVert z\rrVert
^2 \bigr)\qquad\mbox{for all }z\in\R^n.
\end{equation}
Because $\Upsilon_{\pi\circ X}(\beta)<\infty$, by assumption, it follows
that $\int_{\R^n}(\beta+\llVert z\rrVert ^2)^{-1} \,\d z<\infty$ and hence
$n=1$.\footnote{This illustrates, in the present setting,
the well-known folklore fact that the SHE does not have a
mild solution as a function on $\R^n$ when $n\geq2$; see Dalang \cite{Dalang}
and Peszat and Zabczyk \cite{PZ}.}
This proves our earlier assertion (\ref{eqClaim}).

Now that we have (\ref{eqClaim}), we know that $G\cong\R\times K$
for a compact Abelian group $K$. Because of Theorem~\ref{thgroupinv},
we may assume, without loss of generality, that our LCA group $G$
is in fact equal to $\mathbf{R}\times K$.
Thus, thanks to Propositions~\ref{prproj1} and~\ref{prL2LB},
%
%
%e14.15 #&#
\begin{eqnarray}
\E \bigl(\llVert u_t\rrVert_{L^2(\R\times K)}^2 \bigr) &
\ge&\operatorname{const}\cdot\, \Biggl\{ 1+\sum_{j=1}^\infty
\biggl( \frac{\ell_\sigma^2\lambda
^2}{{\e}} \cdot\Upsilon_X(j/t)
\biggr)^j \Biggr\}
\nonumber\\[-8pt]\\[-8pt]\nonumber
&\ge&\operatorname{const}\cdot\, \Biggl\{ 1+\sum_{j=1}^\infty
\biggl( \frac{\ell_\sigma^2\lambda
^2}{{\e}} \cdot\Upsilon_{\pi\circ X}(j/t)
\biggr)^j \Biggr\}.
\end{eqnarray}
According to Bochner's estimate (\ref{eqBochner}),
%
%
%e14.16 #&#
\begin{equation}
\Upsilon_{\pi\circ X}(\beta) \ge\operatorname{const}\cdot\int
_0^\infty\frac{\d x}{\beta+x^2}\ge\frac{\operatorname{const}}{\sqrt
\beta},
\end{equation}
uniformly for all $\beta\ge\beta_0$, for every fixed $\beta_0>0$.
Thus, we may appeal to Lemma~\ref{lemsums}---with $\rho:=\nicefrac12$
and $a=1$---in
order to see that $\E(\llVert u_t\rrVert _{L^2(\R\times K)}^2)\ge
a\exp(b\lambda^4)$,
simultaneously for all $\lambda>c_{1,1/2}$, where $c_{1,1/2}$
is a finite constant that is independent of $\lambda$.
This proves that $\underline{\mathfrak{e}}(t)
\ge4$ when $G$ is noncompact (as well as connected).

We complete the proof of the theorem by proving it when $G$ is compact,
connected, metrizable and has at least 2 elements.

A theorem of Pontryagin (\cite{Morris}, Theorem 33, page~106)
states that if $G$ is a locally connected
LCA group that is also metrizable then
%
%
%e14.17 #&#
\begin{equation}
G\cong\R^n\times\mathbf{T}^m\times D,
\end{equation}
where $0\le n<\infty$, $0\le m\le\infty$, and $D$ is discrete. Of course,
$\mathbf{T}^\infty:=\mathbf{T}\times\mathbf{T}\times\cdots$
denotes the countable direct product of the torus
$\mathbf{T}$ with itself, as is customary.

Since $G$ is compact and connected, we can deduce readily that
$n=0$ and $D$ is trivial; that is, $G\cong\mathbf{T}^m$ for some
$0\le m\le\infty$. Because, in addition, $G$ contains at least
2 elements, we can see that $m\neq0$; thus,
%
%
%e14.18 #&#
\begin{equation}
G\cong\mathbf{T}^m\qquad\mbox{for some }1\le m\le\infty.
\end{equation}
As a matter of fact, the forthcoming argument can be refined to prove
that $m=1$; see our earlier proof of (\ref{eqClaim}) for a model of
such a proof. But since we will not need this fact, we will not
prove explicitly that $m=1$. Suffice it to say that, since $m\ge1$,
an application of Tychonoff's theorem yields
%
%
%e14.19 #&#
\begin{equation}
G\cong\mathbf{T}\times K,
\end{equation}
for a \emph{compact} Hausdorff Abelian group $K$.
Theorem~\ref{thgroupinv} reduces our problem
to the case that $G=\mathbf{T}\times K$, owing to projection.

Let now $\pi$ denote the canonical projection from
$\mathbf{T}\times K$ to $\mathbf{T}$,
and argue as in the noncompact case to see that
%
%
%e14.20 #&#
\begin{equation}
\E \bigl(\llVert u_t\rrVert_{L^2(\mathbf{T}\times K)}^2 \bigr)
\ge\operatorname{const}\cdot\, \Biggl\{ 1+\sum_{j=1}^\infty
\biggl( \frac{\ell_\sigma^2\lambda
^2}{{\e}} \cdot\Upsilon_{\pi\circ X}(j/t)
\biggr)^j \Biggr\}.
\end{equation}
Bochner's estimate (\ref{eqBochner}) has the following analogue for
the L\'evy process $\pi\circ X$ on $\mathbf{T}$: there exists $A\in
(0,\infty)$
such that
%
%
%e14.21 #&#
\begin{equation}
\Re\Psi_{\pi\circ X}(n)\le A \bigl(1+n^2 \bigr)\qquad\mbox{for all
}n\in\mathbf{Z}.
\end{equation}
[The proof of this bound is essentially the same as the proof
of (\ref{eqBochner}).] Since the dual to $\mathbf{T}$
is $\mathbf{Z}$, it follows that
%
%
%e14.22 #&#
\begin{equation}
\Upsilon_{\pi\circ X}(\beta) =\operatorname{const}\cdot\,\sum
_{n=-\infty}^\infty\frac{1}{\beta+\Re\Psi(n)} \ge\frac{\operatorname
{const}}{\sqrt
\beta},
\end{equation}
uniformly for all $\beta\ge\beta_0$, for every fixed $\beta_0>0$. A final
appeal to Lemma~\ref{lemsums}---with $\rho:=\nicefrac12$---completes
the proof.
\end{pf*}

\begin{pf*}{Proof of Theorem~\ref{thconnected2}}
Consider the special case that $G=\R$ and $X$ is a symmetric
stable L\'evy process with index $\alpha\in(0,2]$; that is,
$\Psi(\xi)=\llvert\xi\rrvert^\alpha$. Condition (\ref{D}) holds if and only if
$\alpha\in(1,2]$, a condition which we now assume. The generator
of $X$ is the fractional Laplacian $\mathscr{L}:=-(-\Delta)^{\alpha
/2}$ on $\R$.
A direct computation reveals that
%
%
%e14.23 #&#
\begin{equation}
\Upsilon(\beta) = \operatorname{const}\cdot\int_0^\infty
\frac{\d
x}{\beta+x^\alpha} =\operatorname{const}\cdot\,\beta^{-(\alpha-1)/\alpha}.
\end{equation}
In particular, for every $\varepsilon\in(0,1)$,
%
%
%e14.24 #&#
\begin{equation}
\Upsilon^{-1} \biggl( \frac{1}{(1+\varepsilon)^2\lambda^2\lip
^2} \biggr) \le
\operatorname{const}\cdot\,\lambda^{2\alpha/(\alpha-1)}\qquad\mbox{for all }\lambda>1.
\end{equation}
This yields
%
%
%e14.25 #&#
\begin{equation}
\overline{\mathfrak{e}}(t) \le\frac{2\alpha}{\alpha-1},
\end{equation}
in this case; see the proof of the first portion of Theorem~\ref{thdiscrete}
for more details. And an appeal to Lemma~\ref{lemsums} yields
%
%
%e14.26 #&#
\begin{equation}
\underline{\mathfrak{e}}(t)\ge\frac{2\alpha}{\alpha-1}.
\end{equation}
See the proof of Theorem~\ref{thconnected1} for some details.

Thus, for every $\alpha\in(1,2]$, we have found a model whose
noise excitation index~is
%
%
%e14.27 #&#
\begin{equation}
\label{this} \mathfrak{e} = \frac{2\alpha}{\alpha-1}.
\end{equation}
Since $\theta:=2\alpha/(\alpha-1)$ can take any value in $[4,\infty)$,
as $\alpha$ varies in $(1,2]$,
equation~(\ref{this}) proves the theorem.
\end{pf*}

\section*{Acknowledgements}
We have benefitted a great deal from discussions with Professors Daniel
Ciubotaru
and Dragan Milicic, whom we thank heartily. Many thanks are due to
Professor Martin Hairer,
whose insightful remark about the exponential martingale led to parts
of the
present work. The paper has benefitted greatly from four detailed
reports by
anonymous referees and Associate Editors. We thank them all.

%\begin{appendix}
%\section{}
%\end{appendix}

%\section*{Acknowledgments}

%\begin{supplement}[id=suppA]
%\slink[doi]{10.1214/00-AOPXXXXSUPP} %[doi,text={...}] - jei reikia
%suskaldyti doi
%\sfilename{aopXXXX\_supp.pdf}
%\end{supplement}

% imsref loaded by linak, 2014-07-15 16:08:25
% imsref loaded by linak, 2014-07-15 16:36:47
% imsref loaded by linak, 2014-07-15 16:43:14
% imsref loaded by linak, 2014-07-15 16:52:07
%
% imsref loaded by linak, 2014-07-16 11:14:18
% imsref loaded by linak, 2014-07-17 13:13:34

\printaddresses
\end{document}